\UseRawInputEncoding 
\documentclass[12pt, oneside]{amsart}
\usepackage{amsmath}
\usepackage{amssymb}
\usepackage{color}
\usepackage[usenames,dvipsnames,x11names,svgnames]{xcolor}
\usepackage[colorlinks=true,linkcolor=NavyBlue,citecolor=DarkGreen, urlcolor=blue]{hyperref}
\usepackage{amsthm}
\usepackage{amscd}
\usepackage{geometry}
\usepackage{mathrsfs}
\usepackage{dsfont}
\usepackage{mathtools}

\newtheorem{thm}{Theorem}

\newtheorem{conj}{Conjecture}
\newtheorem{prop}{Proposition}
\newtheorem{lem}[prop]{Lemma}
\newtheorem{cor}[thm]{Corollary}

\theoremstyle{definition}

\newcommand{\mb}{\mathbb}
\newcommand{\mc}{\mathcal}

\newcommand{\ol}{\overline}
\newcommand{\ul}{\underline}
\newcommand{\leqs}{\leqslant }
\newcommand{\geqs}{\geqslant }

\newcommand{\1}{\mathds{1}}
\newcommand{\norm}[1]{\lVert #1\rVert}
\newcommand{\normh}{\norm{h\alpha}}

\newcommand{\near}[1]{\lfloor #1\alpha\rceil}

\newcommand{\WIt}{W_{\mc{I},t}}
\newcommand{\WJt}{W_{\mc{J},t}}

\newcommand{\wit}{w^{\mc{I}}_t}
\newcommand{\wjt}{w^{\mc{J}}_t}

\newcommand{\be}{\begin{equation*}}
\newcommand{\ee}{\end{equation*} }

\newcommand{\ben}{\begin{equation}}
\newcommand{\een}{\end{equation} }

\newcommand{\bs}{\begin{split}}
\newcommand{\es}{\end{split}}

\newcommand{\bmu}{\begin{multline*}}
\newcommand{\emu}{\end{multline*}}

\newcommand{\bmun}{\begin{multline}}
\newcommand{\emun}{\end{multline}}

\newcommand{\R}{\mathbb{R}}
\newcommand{\Z}{\mathbb{Z}}

\begin{document}

\author[W. Heap]{Winston Heap}
\email{winstonheap@gmail.com}

\author[A. Sahay]{Anurag Sahay}
\email{anuragsahay@purdue.edu}

\title[The fourth moment of the Hurwitz zeta function]{The fourth moment of the Hurwitz zeta function}

\begin{abstract} We prove a sharp upper bound for the fourth moment of the Hurwitz zeta function $\zeta(s,\alpha)$ on the critical line when the shift parameter $\alpha$ is irrational and of irrationality exponent strictly less than $3$. As a consequence, we determine the order of magnitude of the $2k$th moment for all $0 \leqs k \leqs 2$ in this case. In contrast to the Riemann zeta function and other $L$-functions from arithmetic, these grow like $T (\log T)^k$. This suggests, and we conjecture, that the value distribution of $\zeta(s,\alpha)$ on the critical line is Gaussian.  \end{abstract}

\maketitle

\section{Introduction}

Let $0 < \alpha \leqs 1$ be a fixed shift parameter. The Hurwitz zeta function is defined  by 
\[ \zeta(s,\alpha) = \sum_{n\geqs 0} \frac{1}{(n+\alpha)^s} \]
for $\Re(s) > 1$, and can be extended to a meremorphic function on $\mathbb{C}$ with a simple pole at $s = 1$. In this paper, we are interested in the value distribution of $\zeta(s,\alpha)$ on the critical line. As we shall see, this can depend heavily on the Diophantine nature of the shift parameter $\alpha$.

The case of rational $\alpha$ is relatively well understood. When $\alpha = 1$, we get back the usual Riemann zeta function: $\zeta(s,1)=\zeta(s)$, and an easy calculation shows that $\zeta(s,\tfrac12) = (2^s - 1) \zeta(s)$. 
When $\alpha=a/q$  with $q > 2$, the orthogonality of Dirichlet characters implies that
\begin{equation}\label{rational hurwitz} 
\zeta(s,\tfrac{a}{q}) = \frac{q^s}{\varphi(q)}\sum_{\chi \bmod q} \ol{\chi(a)} L(s,\chi), 
\end{equation}
giving a connection to Dirichlet $L$-functions. 
In these cases, the typical fluctuations of the underlying $L$-functions at large height are determined in accordance with Selberg's central limit theorem \cite{Sel, Tthesis}. This gives a log-normal distribution: for fixed $V$, 
\[
\frac{1}{T}\mathrm{meas}\bigg(\bigg\{t\in [T,2T]: \frac{\log|\zeta(\tfrac12+it)|}{\sqrt{\tfrac12\log\log T}}\geqs V\bigg\}\bigg)\sim \frac{1}{\sqrt{2\pi}}\int_V^\infty e^{-x^2/2}dx
\]
as $T\to\infty$ where $\mathrm{meas}$ denotes Lebesgue measure.

Going beyond typical fluctuations, larger values are determined by the moments \cite{Rdev}. Here, the Keating--Snaith conjecture \cite{keatingsnaithzeta} predicts that for real $k>0$
\[
\frac{1}{T}\int_T^{2T}|\zeta(\tfrac12+it)|^{2k}dt\sim c_k(\log T)^{k^2} 
\]  
with an explicitly given constant $c_k$ (see also \cite{CFKRS, CG, CGk, DGH}). This conjecture is only known in the classical cases of $k=1,2$ due to Hardy-Littlewood \cite{HL} and Ingham \cite{I}, respectively -- in general the conjecture is wide open. However, recently a great deal of progress has been made on the order, especially under the assumption of the Riemann Hypothesis (RH). It is now known that 
\[
\frac{1}{T}\int_T^{2T}|\zeta(\tfrac12+it)|^{2k}dt\asymp (\log T)^{k^2} 
\]  
for all real $k\geqs 0$ on RH \cite{Ha,HSlower, Sound} (and unconditionally for $0\leqs k\leqs 2$, \cite{HRSupper}). Although we have stated these results for the Riemann zeta function, they are known (or expected) to hold for any `reasonable' $L$-function i.e. those coming from arithmetic and in the Selberg class, say. Furthermore, finite sets of $L$-functions that satisfy Selberg's orthonormality conjecture behave independently at the scale of the central limit theorem \cite{bombhej,selbergoldandnew}, but conjecturally show slight dependence coming from the Euler product at the scale of moments \cite{HeapMPCPS}. Thus, for such $L$-functions (in particular for Dirichlet $L$-functions), the joint value distribution on the critical line is well understood, at least conjecturally. This gives satisfactory answers for distributional questions about the Hurwitz zeta function for rational shifts $\alpha$, as carried out in \cite{Sahay}.
  
When $\alpha$ is irrational, the behaviour of the Hurwitz zeta function is more mysterious and the lack of connection with $L$-functions allows many peculiarities to arise. Perhaps most notable -- and in fact this also holds for any rational $\alpha\neq 1,\tfrac12$ -- is that there is no Euler product and consequently the Riemann Hypothesis fails quite spectacularly. It is classical \cite{cassels1961footnote,davenport1936zeros} that for any fixed $\delta > 0$ and $\alpha \neq 1,\tfrac12$, $\zeta(s,\alpha)$ vanishes infinitely often in the strip $1 < \Re(s) < 1+\delta$ and a similar result is known for substrips of the critical strip \cite{gonekthesis,Vzeros} when $\alpha$ is transcendental or rational (see \cite{minealgirr} for recent progress on the difficult case of algebraic irrational).

Nevertheless, since $\zeta(s,\alpha)$ possesses a functional equation one may expect various analytic aspects to be shared with the usual $L$-functions. In particular, one may expect a Lindel\"of Hypothesis to hold up to, and on, the critical line \cite{GS}. It is therefore of interest to see how the lack of Euler product affects large values of $\zeta(\tfrac12+it,\alpha)$ for individual\footnote{Much of the past literature \cite{alphamean1,alphamean5, alphamean2,alphamean3,alphamean8, alphamean4,alphamean7,alphamean6} on the Lindel\"of Hypothesis and mean values for $\zeta(s,\alpha)$ focuses on taking a mean-value over $\alpha$ for fixed $s$, a topic we do not pursue here.} irrational $\alpha$. 

Much work has gone into understanding the value distribution in the strip $1/2<\sigma<1$. Here, it is known that for all $\alpha$, $\zeta(s,\alpha)$ possesses a limiting distribution for large $t$ \cite{GL,GLbook,M} -- an aspect shared by the usual zeta function. An analogue of Voronin's famous universality theorem for $\zeta(s)$ is also known to hold here when $\alpha$ is rational or transcendental \cite{bagchi,gonekthesis}. Considerable effort has gone into aspects of this universality; see \cite{anderssonarxiv,univ1,GLbook,minealgirr,univ2} for a slice of the relevant literature. Unfortunately, universality of $\zeta(s,\alpha)$ for a fixed algebraic irrational $\alpha$ remains open, although interesting progress has been made recently \cite{minealgirr, SS} (see also \cite{anderssonarxiv} for discussions about this difficult problem).

On the critical line much less is known and the distribution is undetermined for irrational $\alpha$. However, some moments are accessible. A result of Rane \cite{Rane} states that for all $0<\alpha\leqs 1$, 
\begin{equation}\label{rane}
\int_0^T |\zeta(\tfrac12+it,\alpha)|^2dt
=
T\log T+T(c(\alpha)+\gamma-1-\log2\pi)+O(T^{1/2}\log T)
\end{equation}
where 
\[
c(\alpha)=\lim_{N\to\infty}\bigg(\sum_{n\leqs N}\frac{1}{n+\alpha}-\log N\bigg).
\]
See also \cite{Nar, Tao} for further improvements. Note that the above formula does not depend on any Diophantine properties of $\alpha$ and holds for  rationals and irrationals alike.  

Naturally, the fourth moment is more difficult. For rational $\alpha$, on utilising formula \eqref{rational hurwitz} the leading term in the fourth moment was given by the second author in \cite{Sahay} (see also \cite{anderssonarxiv}). Here it was shown that for $1\leqs a< q$ with $(a,q)=1$ and fixed $q\geqs 3$,
\begin{equation}\label{rational fourth}
\int_0^T |\zeta(\tfrac12+it,\tfrac{a}{q})|^4dt
\sim
\frac{T(\log T)^4}{2\pi^2q}\prod_{p|q}\Big(1-\frac{1}{p+1}\Big).
\end{equation}
This is smaller than the fourth moment of the Riemann zeta function 
which suggests a greater degree of cancellation. Furthermore, the main term here fluctuates very heavily depending on the denominator $q$, even for rationals that are close to each other. One may wonder then, with a view towards rational approximations of irrational $\alpha$, about the uniformity of this result with respect to $q$ and if the moments are \emph{much} smaller when $\alpha$ is irrational. Our main result shows that this is indeed the case. 

\begin{thm}\label{main thm} Suppose $0<\alpha<1$ is irrational with irrationality exponent $\mu(\alpha)<3$. Then for large $T$ we have 
\begin{equation*}
\int_T^{2T}|\zeta(\tfrac12+it,\alpha)|^4dt
\ll
T(\log T)^2.
\end{equation*}
\end{thm}

We recall that the \emph{irrationality exponent} $\mu(\alpha)$ of a real number $\alpha$ is defined as the supremum of the set of real numbers $\mu$ for which there exist infinitely many rational numbers $p/q$ such that 
\[0<|\alpha-p/q|<1/q^\mu.\]
The irrationality exponent of rationals is $1$ whereas $\mu(\alpha)\geqs 2$ for irrational $\alpha$. A famous theorem of Roth \cite{roth} gives that $\mu(\alpha)=2$ for all algebraic irrationals whilst many transcendental numbers also have $\mu(\alpha)=2$. Indeed, in the Lebesgue sense almost all real numbers have $\mu(\alpha)=2$ and so, in particular, Theorem \ref{main thm} holds for generic irrational $\alpha$.
The Liouville numbers provide examples of reals with $\mu(\alpha)=\infty$ and there exist many\footnote{admittedly Lebesgue measure zero, but of full Hausdorff dimension.} so-called \emph{very well approximable numbers} satisfying $\mu(\alpha)>2$, see \cite{HV, LSV}. 


Since the fourth moment is proportional to the second moment squared, an application of H\"older's inequality allows one to immediately determine the order of all lower moments. 

\begin{cor}\label{main cor}Suppose $0<\alpha<1$ is irrational with $\mu(\alpha)<3$. Then for real $0\leqs k\leqs 2$ and large $T$ we have 
\begin{equation*}
\int_T^{2T}|\zeta(\tfrac12+it,\alpha)|^{2k}dt
\asymp
T(\log T)^k.
\end{equation*}
\end{cor}

Theorem~\ref{main thm} and Corollary~\ref{main cor} may seem surprising from two perspectives. For those familiar with moment problems for $L$-functions from arithmetic, these results suggest that the Hurwitz zeta function is not log-normal, unlike the Riemann zeta function and other $L$-functions. A moment's reflection will convince the reader, however, that this is not too surprising: the Hurwitz zeta function has no Euler product, and the Euler product is crucially important for log-normality. On the other hand, those familiar with the Hurwitz zeta function may be surprised that these results are able to deal with all algebraic irrationals (as they have $\mu(\alpha) = 2$) but are not able to deal with all transcendental numbers (as those with $\mu(\alpha) \geqs 3$ are not covered by our theorem). This is in contrast to work on Voronin universality or zero distributions, where the rational and transcendental cases are much easier, while the algebraic irrational case is harder (and, indeed, often open; see, for example, \cite{anderssonarxiv,minealgirr, SS}). As we explain further at the end of this introduction, the reason behind our restriction on $\alpha$ is essentially that the off-diagonal terms are difficult to control with sufficient uniformity when $\alpha$ is very well approximable. 

Theorem~\ref{main thm} naturally raises the question of asymptotics. We have some speculations in this regard. For transcendental $\alpha$, the numbers $\log(n+\alpha)$ are linearly independent which suggests that the $(n+\alpha)^{-it}$ may act as independent random variables. This would entail Gaussian behaviour and accordingly one could expect the fourth moment to be $\sim 2 T(\log T)^2$, in view of Rane's result \eqref{rane}. One can interpret this expectation as saying that the leading term in the fourth moment arises from what are essentially diagonal terms in the approximate functional equation. Indeed, if the $(n+\alpha)^{-it}$ were purely orthogonal then we would expect 
\begin{equation}\label{naive asym}
\int_T^{2T}|\zeta(\tfrac12+it,\alpha)|^{4}dt
\sim T\sum_{\substack{(n_1+\alpha)(n_2+\alpha)=\\(n_3+\alpha)(n_4+\alpha)\\n_j\leqs T}}\frac{1}{\sqrt{(n_1+\alpha)(n_2+\alpha)(n_3+\alpha)(n_4+\alpha)}}.
\end{equation}
For irrational $\alpha$ the solutions of the given equation are simply the diagonals\footnote{Here we should distinguish between the ``harmonic" diagonals $(n_1+\alpha)(n_2+\alpha)=(n_3+\alpha)(n_4+\alpha)$ which lead to a non-oscillating sum and the ``Diophantine" diagonals which are the permuted solutions $\{n_1,n_2\} = \{n_3,n_4\}$ of this equation.} $n_1=n_3,\,n_2=n_4$ and $n_1=n_4,\,n_2=n_3$ which leads to the asymptotic 
\[
2T\bigg(\sum_{n\leqs T}\frac{1}{n+\alpha}\bigg)^2\sim 2T(\log T)^2.
\]

Similarly, for higher moments one may expect that
\begin{equation}\label{higher}
\int_T^{2T}|\zeta(\tfrac12+it,\alpha)|^{2k}dt
\sim T\sum_{\substack{\ul{n}\in D_k(T)}}\frac{1}{{\prod_{j=1}^{2k}(n_j+\alpha)^{1/2}}}
\end{equation}
where 
\[
D_k(T)
=
\bigg\{(n_j)_{j=1}^{2k}\in\mathbb{N}^{2k}: \prod_{i=1}^k(n_i+\alpha)=\prod_{j=k+1}^{2k}(n_j+\alpha),\,\,n_j\leqs T\bigg\}.
\]
For $\alpha$ transcendental or algebraic of degree $d\geqs k$ it is fairly easy to see that $D_k(T)$ consists of just permutations, giving $|D_k(T)|\sim k!\,T^k$. When $d<k$, there may exist solutions not arising from permutations, however, in \cite{HSW} (see also \cite{scoopers}) it was shown that these contribute $\ll T^{k-d+1+\epsilon}$ and so one retains this asymptotic for $|D_k(T)|$. The effect of these extra solutions on the weighted sum in \eqref{higher} is not clear, and the behaviour of the off-diagonals in higher moments is also not clear, especially with regards to dependence on $\mu(\alpha)$. Based on these considerations we make the following conjecture. 

\begin{conj} \label{main conj} Let $k\in\mathbb{N}$ and $0<\alpha\leqs 1$ be an irrational number. Then for algebraic $\alpha$ of degree $d\geqs k$ and transcendental $\alpha$ satisfying $\mu(\alpha)=2$ we have
\[
\int_T^{2T}|\zeta(\tfrac12+it,\alpha)|^{2k}dt
\sim
k!\, T(\log T)^k
\] 
as $T\to\infty$.
\end{conj}
In particular, we expect the asymptotic to hold for almost all $\alpha$ in the measure-theoretic sense. If one makes an additional average in the $\alpha$-aspect (concretely, over $\alpha \in [1,2]$), then the above was conjectured by Andersson in an unpublished part of his thesis \cite{Athesis}, where he also proved his conjecture for $k = 2$. Thus, Theorem~\ref{main thm} can be seen as removing the averaging in Andersson's work, albeit at the cost of showing only an upper bound instead of an asymptotic. We discuss the difficulties in proving the $k = 2$ case of Conjecture~\ref{main conj} in \S\ref{sec: issues}.

Since Conjecture~\ref{main conj} speculates that $\zeta(s,\alpha)$ has the moments of a complex Gaussian with mean $0$ and variance $\log T$, we acquire the following conjecture regarding the distribution of $\zeta(s,\alpha)$ on the critical line.

\begin{conj} \label{Gauss conj} Let $S\subset\mathbb{C}$ be Borel. Then for large $T$ and almost all $0<\alpha<1$ we have
\[
\frac{1}{T}\mathrm{meas}\Big\{t\in[T,2T]: \frac{\zeta(\tfrac12+it,\alpha)}{\sqrt{\log T}}\in S\Big\}
\sim
\frac{1}{{2\pi}}\iint_S e^{-(x^2+y^2)/2}dxdy. 
\] 
\end{conj}

Thus, unlike the case for $\zeta(s,1) = \zeta(s)$, we expect $\zeta(s,\alpha)$ to be normal instead of log-normal for generic $\alpha$. It would be interesting to understand, at least conjecturally, the finer aspects of the distribution such as large deviations and maximal values. For the Riemann zeta function, the behaviour of the tails are unclear although it has been speculated that they remain (essentially) log-normal up to the maximum \cite{AHZ} which, through using various random models, has been conjectured \cite{FGH} to satisfy 
\[
\max_{t\in[T,2T]}|\zeta(\tfrac12+it)|=\exp((1+o(1))\sqrt{\tfrac12\log T\log\log T})
. 
\]
The random matrix theory and random Euler product models of \cite{FGH} do not appear to have any immediate analogues for the Hurwitz zeta function (this is certainly true of the latter model when $\alpha$ is irrational). For such questions, it may be more appropriate to model the $(n+\alpha)^{-it}$ by independent random variables (in a similar fashion to \cite{AHZ}) however the robustness of this model is unclear at the extremes. At any rate, given the distribution in Conjecture \ref{Gauss conj} it seems possible that the true maximum of the Hurwitz zeta function could be markedly smaller than that of the Riemann zeta function for generic $\alpha$; it would be interesting to make this precise.

There exist several cousins of the Hurwitz zeta function which have also been the subject of intensive study in their value distribution. 
Most notable are the dual to the Hurwitz zeta function, namely the periodic zeta function,
\[P(s,\lambda) = \sum_{n\geqs 1} \frac{e(\lambda n)}{n^s}, \]
and the more general Lerch zeta function 
\[
L(s,\alpha,\lambda)=\sum_{n\geqs 0}\frac{e(\lambda n)}{(n+\alpha)^s}
\] 
where as usual $e(x)=e^{2\pi i x}$ (for an overview, see \cite{GLbook}). Again, these are defined initially in $\Re(s) > 1$ and extend to functions on $\mathbb{C}$ that are holomorphic except for a possible simple pole at $s = 1$ which only occurs if $\lambda \in \Z$. The periodic zeta function $P(s,\alpha)$ is related to $\zeta(s,\alpha)$ by a functional equation; see \eqref{hurwitz func} and, as with $\zeta(s,\alpha)$, it can be related to $L$-functions when $\alpha$ is a rational with no such relation likely when $\alpha$ is irrational. Through similar considerations to the above, on the critical line one might expect Gaussian-like behaviour for these functions also. In fact, by a short argument with the functional equation, the $2k$th moment of $P(1/2+it,\alpha)$ is equal to that of $\zeta(1/2+it,\alpha)$ up to a negligible error. Theorem~\ref{main thm} thus also implies the order of magnitude for low moments of $P(s,\alpha)$.

\begin{cor}\label{main cor 2}Suppose $0<\alpha<1$ is irrational with $\mu(\alpha)<3$. Then for $0\leqs k\leqs 2$ and large $T$ we have 
\begin{equation*}
\int_T^{2T}|P(\tfrac12+it,\alpha)|^{2k}dt
\asymp
T(\log T)^k.
\end{equation*}
\end{cor}

Previously, the fourth moment of $P(s,\alpha)$ in the strip $1/2<\sigma<1$ was computed for \emph{all} irrational $\alpha$ by Laurin\v{c}ikas--\v{S}iau\v{c}i\=unas \cite{LS} where an asymptotic with leading term $c(\sigma,\alpha)T$ was given. The constant $c(\sigma,\alpha)$ implicitly contains some Diophantine information on $\alpha$, although the influence is mild. Clearly, the distinction between rational and irrational cases becomes more pronounced on the half-line (it follows from \eqref{rational fourth} and \eqref{h to p} that for rational $\alpha=a/q$ the fourth moment on the critical line is $\sim c_qT(\log T)^4$). 
It is likely that our methods also generalise to the Lerch zeta function under irrationality and Diophantine assumptions on either $\alpha$ or $\lambda$, in which case we would similarly expect moments of size $T(\log T)^k$ and a Gaussian distribution. Currently, the second moment is known for $1/2\leqs \sigma<1$ and all $0<\alpha,\lambda<1$ due to Garunk\v{s}tis--Laurin\v{c}ikas--Steuding \cite{GLS}.

We end this introduction with an explanation of the Diophantine assumptions on $\alpha$ in our results. Since the harmonics $(n+\alpha)^{-it}$ are not truly orthogonal at our scale, \eqref{naive asym} is an oversimplification. In particular, when dealing with the off-diagonal terms in the mean value
\[
\int_T^{2T}\bigg|\sum_{n\leqs T}\frac{1}{(n+\alpha)^{1/2+it}}\bigg|^4dt,
\]
it can happen that $(n_1 + \alpha)(n_2 + \alpha) \neq (n_3 + \alpha)(n_4 + \alpha)$ but that $(n_1 + \alpha)(n_2 + \alpha) - (n_3 + \alpha)(n_4 + \alpha)$ is still very close to zero. In the case of the usual zeta function ($\alpha=1$) this difference is at worst 1, and for $\alpha=a/q$ at worst $1/q$, but for irrational $\alpha$ it can be arbitrarily close to zero. Thus, there are potentially many off-diagonal terms with no oscillation, each term giving a contribution of size $\approx \int_T^{2T} 1dt=T$.
These worst cases occur when 
\[ \alpha \approx \frac{n_1 n_2 - n_3 n_4}{n_3 + n_4 - n_1 - n_2} \]
and so in order for them to not contribute a main term, we require $\alpha$ to not have too many good quality rational approximations. This naturally leads to assumptions on $\mu(\alpha)$. It is unclear whether the particular condition $\mu(\alpha) < 3$ is a technical limitation of our method or whether one should expect genuinely different behaviour for certain $\alpha$ which have larger irrationality exponent, $\mu(\alpha)=3$ or 4, say. 

When dealing with the off-diagonals, we will use our Diophantine assumptions on $\alpha$ to bound sums of reciprocals of fractional parts such as
\[ \sum_{n\leqs N} \frac{1}{\norm{n\alpha}^{\eta}},\]
for $0 < \eta < 1$ which will arise naturally upon throwing away some possible cancellation. The order of these sums was worked out comprehensively by Kruse \cite{K}, and this indicates that weakening or removing the assumption that $\mu(\alpha) < 3$ will require a different treatment that exploits this cancellation that we neglected.

\subsubsection*{Notation} We use aysmptotic notation $\ll, \gg, O(\cdot), o(\cdot), \sim,\asymp$ that is standard in analytic number theory. In particular, we use the Vinogradov notation $A\ll B$ interchangeably with the Bachmann-Landau notation $A = O(B)$ for the inequality $|A| \leqs CB$, where $C$ is some large unspecified constant and $B$ is positive. The quantities $A$ and $B$ will typically depend on some parameters and the range in which the inequalities hold should be clear from context. Asymptotic statements may involve an arbitrarily small parameter $\epsilon>0$, which may vary from occurrence to occurrence. To distinguish arbitrary epsilons from fixed ones, we shall use $\varepsilon$ for the latter in the proof of Lemma~\ref{expsum lem}. Implicit constants may depend on $\alpha$ throughout and $\epsilon$ wherever it appears, but will be uniform in other parameters unless specified via subscripts. The parameter 
\begin{equation*}
\tau:=\sqrt{t/2\pi}
\end{equation*}
will appear throughout the paper. Finally we use the standard notation for additive characters, $e(t) := e^{2\pi it}$.

\subsection*{Acknowledgements} We thank Steve Gonek and Trevor Wooley for useful discussions on these topics. AS is partially supported through Purdue University start-up funding available to Trevor Wooley.


\section{Preliminaries on Diophantine approximation} \label{sec:diophantine}


We first recall a few basic properties of continued fraction expansions. These can be found in many sources e.g. see \cite{Cassels, Lang}. Let \begin{equation*} \alpha = [a_0;a_1,a_2,\cdots] = a_0 + \cfrac{1}{a_1 + \cfrac{1}{a_2 + \cfrac{1}{\ddots}}},\end{equation*} 
be the continued fraction expansion of $\alpha$, and let
\begin{equation*} \frac{p_k}{q_k} = [a_0;a_1,\cdots,a_k],\end{equation*}
be the principal convergents of $\alpha$.  We call $a_k$ the partial quotients of $\alpha$. An induction gives the following recursions
\begin{equation*}\label{induction}
\begin{split}
 p_{k+1}=&a_{k+1}p_k+p_{k-1} 
\\
 q_{k+1}=&a_{k+1}q_k+q_{k-1} 
\end{split}
\end{equation*}
and note the second of these implies $a_{k+1}\leqs q_{k+1}/q_k$. 

A consequence of Dirichlet's approximation theorem is that for irrational $\alpha$ there exist infinitely many $p,q$ such that 
\begin{equation*}\label{Dirichlet}
\Big|\alpha-\frac{p}{q}\Big|<\frac{1}{q^2}. 
\end{equation*}
The principal convergents form a sequence of best possible approximations to $\alpha$ in the sense that $q_k$ is the smallest integer $q>q_{k-1}$ such that $\|q\alpha\|<\|q_{k-1}\alpha\|$ where here and throughout 
\[
\norm{x} = \min_{n\in\Z}|x-n|
\]
is the distance of $x$ to the nearest integer. They also satisfy the inequalities 
\[
\frac{1}{2q_kq_{k+1}}<\Big|\alpha-\frac{p_k}{q_k}\Big|<\frac{1}{q_kq_{k+1}} 
\]
from which the following formula for the irrationality exponent can be deduced
\begin{equation*}\label{irr exp}
\mu(\alpha)=1+\limsup_{k\to\infty}\frac{\log q_{k+1}}{\log q_k}.
\end{equation*}
See Theorem 1 of \cite{Sondow}. Note that our condition $\mu(\alpha)<3$ implies there exists some $\delta$ such that for large $q_k$,
\begin{equation}\label{assump ineq}
q_{k+1}\ll q_k^{2-\delta}.
\end{equation}
In other words, the denominators in our rational approximations do not grow too rapidly. This will be the form in which we make use of our condition $\mu(\alpha)<3$.




We now state our main lemma on sums of reciprocals of fractional parts. The order of these sums was worked out fairly comprehensively by Kruse \cite{K} although we can also recommend the memoir \cite{BHV} for an exposition of this topic and further interesting results including explicit bounds.  

\begin{lem} \label{kruse} Suppose $\alpha$ is irrational with $\mu(\alpha)<3$. Then, 
\begin{equation*} \sum_{h\leqs N} \frac{1}{\normh^{1/2}} \ll N \end{equation*}
\end{lem}

\begin{proof}

This follows from \cite[Theorem 1]{K}. Specifically, plugging $t = 1/2$ in \cite[Equation 73]{K}, and rewriting in our notation, one has
\begin{equation*} \sum_{h\leqs N} \frac{1}{\normh^{1/2}} \asymp N + N^{1/2} a_{K+1}^{1/2}, \end{equation*}
where $K=K(N)$ is the largest integer such that $q_K\leqs N<q_{K+1}$. But $a_{K+1}\leqs q_{K+1}/q_K\ll q_K\leqs N$ with the second inequality following by \eqref{assump ineq}.  The result then follows. 
\end{proof}

We will also be required to demonstrate cancellation in the bilinear sum
\begin{equation*}\sum_{h\leqs y}\sum_{k\leqs x}e(-\alpha d hk),\end{equation*}
where any power savings in both $x$ and $y$ will suffice. To this end, we have the following.

\begin{lem} \label{expsum lem}
Suppose $\alpha$ is irrational with $\mu(\alpha) \leqs 3-\delta$ for a given $\delta>0$. Then,
\begin{equation}\label{square root bound}\sum_{h\leqs x}\sum_{k\leqs y}e(-\alpha d hk) \ll d(xy)^{1 - \delta/5}.\end{equation}
uniformly in $d\geqs 1$. 
\end{lem}

\begin{proof}
In this proof, every occurrence of $\varepsilon>0$ is the same and sufficiently small. At the end, we will choose a value for $\varepsilon$. Without loss of generality we may assume $x\leqs y$. For $\varepsilon>0$ let $\ell=\ell(xy,\varepsilon)$ be the minimal integer such that $q_\ell>(xy)^{1/2+\varepsilon}$. We note that 
\[
x<(xy)^{1/2+\varepsilon}<q_\ell \ll q_{\ell-1}^{2-\delta}\ll (xy)^{1-\delta/2+2\varepsilon};
\]
by \eqref{assump ineq} and then the minimality of $\ell$. 

Now,
\[
|\alpha-\tfrac{p_\ell}{q_\ell}|<\tfrac{1}{q_\ell^2}<\tfrac{1}{(xy)^{1+2\varepsilon}}
\]
and hence 
\[
\sum_{h\leqs x}\sum_{k\leqs y}e(-\alpha d hk)
=
\sum_{h\leqs x}\sum_{k\leqs y}e(- d hk \tfrac{p_\ell}{q_\ell})+O(d(xy)^{1-2\varepsilon})
\]
since $|e(X)-1|\ll X$ uniformly in $X$. 
The latter sum may be decomposed as 
\begin{align*}
\bigg[\sum_{\substack{h\leqs x\\ q_{\ell}|dh}}
+
\sum_{\substack{h\leqs x\\ q_{\ell}\nmid dh}}\bigg]
\sum_{k\leqs y} e(- d hk \tfrac{p_\ell}{q_\ell})
\ll 
y\sum_{\substack{h\leqs x\\ q_{\ell}|dh}}1
+
\sum_{\substack{h\leqs x\\ q_{\ell}\nmid dh}}\frac{1}{\|dhp_\ell/q_\ell\|}.
\end{align*}
The first sum on the right is $\ll {dxy}/{q_\ell}\ll d(xy)^{1/2-\varepsilon}$. Letting $d^*=d/(d,q_\ell)$ and $q_\ell^*=q_\ell/(d,q_\ell)$ the second sum may be written as 
\[
\sum_{\substack{h\leqs x\\ q_{\ell}\nmid dh}}\frac{1}{\|h p_\ell d^*/q_\ell^*\|}
=
\sum_{g|q_\ell^*}\sum_{\substack{h\leqs x\\ (h,q_\ell^*)=g\\q_{\ell}\nmid dh}}\frac{1}{\|(h/g)\cdot \tfrac{p_\ell d^*}{q_\ell^*/g}\|}
\]
and as $h/g$ varies between 1 and $q_\ell^*/g$ we see that each congruence class modulo  $q_\ell^*/g$ is hit at most once by $(h/g)p_\ell d^*$ by coprimality. Therefore, since $x<q_\ell$, this sum is
\[
\ll
\sum_{g|q_\ell^*} \bigg[\mathds{1}_{x\leqs q_\ell^*/g}\sum_{1\leqs n\leqs x} \frac{1}{n/(q_\ell^*/g)}+ \mathds{1}_{x> q_\ell^*/g} \frac{x}{q_\ell^*/g}\sum_{1\leqs n\leqs q_\ell^*/g}\frac{1}{n/(q_\ell^*/g)}\bigg]
\ll
q_\ell d(q_\ell)\log q_\ell,
\]
which is $\ll q_\ell^{1+\varepsilon}\ll (xy)^{1 - \delta/2+3\varepsilon}$ (here, $d(n)$ is the divisor function). Thus, combining the above estimates,
\[
\sum_{h\leqs x}\sum_{k\leqs y}e(-\alpha d hk)
\ll
d(xy)^{1-2\varepsilon} + (xy)^{1-\delta/2 + 3\varepsilon}.
\]
Setting $\varepsilon = \delta/10$, we conclude the lemma.

\end{proof}

\section{Approximate Functional Equation}

In this section we derive a smoothed form of the basic approximate functional equation. Let 
\[\chi(s)=\pi^{1/2-s}\frac{\Gamma((1-s)/2)}{\Gamma(s/2)},\]
which is the factor appearing in the functional equation of the Riemann zeta function $\zeta(s)=\chi(s)\zeta(1-s)$. Note that $|\chi(\tfrac12+it)|=1$. 

\begin{lem}\label{afe lem}
Let $G(z)$ be an entire, even, function satisfying $|G(z)| \ll_C e^{-\Im(z)^2}$ for $\Re(z) \leqs C$ and $\ol{G(z)} = G(\ol{z})$. Then, for any $0< \alpha\leqs 1$ and $t\geqs 1$ we have 
\[
\zeta(\tfrac12+it,\alpha)
=
\sum_{m\geqs 0}\frac{w_t(m+\alpha)}{(m+\alpha)^{1/2+it}}+\chi(\tfrac12+it)\sum_{n\geqs 1}\frac{e(-n\alpha)}{n^{1/2-it}}w_{t}(n)+O(t^{-1/2})
\]
where
\begin{equation*}\label{weight}
w_t(x)
=
\frac{1}{2\pi i }\int_{1-i\infty}^{1+i\infty}({\tau}/{x})^{s}G(s)\frac{ds}{s}.
\end{equation*}
\end{lem} 
\begin{proof}
We start from the contour integral 
\[
I=\frac{1}{2\pi i }\int_{1-i\infty}^{1+i\infty}\zeta(\tfrac12+it+s,\alpha)\tau^s G(s)\frac{ds}{s}.
\]
Shifting the contour to $\Re(s)=-1$ we pick up poles at $s=0$ and $s=-1/2-it$.  Due to the rapid decay of $G(s)$, the latter pole contributes $O(e^{-At})$. Thus,
\begin{equation*}
I
=
\zeta(\tfrac12+it,\alpha)
+
\frac{1}{2\pi i }\int_{-1-i\infty}^{-1+i\infty}\zeta(\tfrac12+it+s,\alpha) \tau^s G(s)\frac{ds}{s}+O(e^{-At}).
\end{equation*}

	In this last integral we let $s\mapsto -s$ and  apply the functional equation for the Hurwitz zeta function \cite[Equation~2]{KR} in the form 
\begin{equation}\label{hurwitz func}
\zeta(1-z,\alpha)= \frac{\chi(1-z)}{2\cos(\tfrac{\pi z}{2})}\Big[e^{-\frac{\pi iz}{2}} P(z,\alpha) + e^{\frac{\pi iz}{2}} P(z,-\alpha)\Big]
\end{equation}
where $P(s,\alpha)$ is the periodic zeta function from the introduction, given by
\begin{equation*} P(s,\alpha) = \sum_{n\geqs 1} \frac{e(n\alpha)}{n^s}, \end{equation*}
when $\Re(s) > 1$. The integral then becomes 

\begin{equation*}
\frac{1}{2\pi i }\int_{1-i\infty}^{1+i\infty}\chi(\tfrac12 + it - s) Q(t,s,\alpha) \tau^{-s} G(s)\frac{ds}{s},
\end{equation*}
where $Q(t,s,\alpha)$ is given by
\begin{equation*}
\frac{\exp(-\tfrac{\pi}{2}(\tfrac{i}{2} + t + is)) P(\tfrac{1}{2} -it + s,\alpha) + \exp(\tfrac{\pi}{2}(\tfrac{i}{2} + t + is)) P(\tfrac{1}{2} - it + s,-\alpha)}{{2\cos(\tfrac{\pi}{2}(\tfrac{1}{2}-it+s))}}. \end{equation*}
 
 Now, for large $t$ and $\Re(s) = 1$, we see that
\begin{equation}\label{periodic zeta bound}
Q(t,s,\alpha) = P(\tfrac{1}{2} - it + s,-\alpha) + O(e^{-\pi t + \pi |s|})\end{equation}
since $P(\tfrac12-it+s,\pm\alpha)$ is bounded in this region and 
\[
\frac{\exp(-\tfrac{\pi}{2}(\tfrac{i}{2} + t + is))}{2\cos(\tfrac{\pi}{2}(\tfrac{1}{2}-it+s))}
\ll e^{-\pi t + \pi |s|},
\qquad 
\frac{\exp(\tfrac{\pi}{2}(\tfrac{i}{2} + t + is))}{2\cos(\tfrac{\pi}{2}(\tfrac{1}{2}-it+s))}
=1+O( e^{-\pi t + \pi |s|})
\]
uniformly in $\Im(s)$. Due to the rapid decay of $G(s)$ and since $\chi(\tfrac{1}{2} +it -s) = O(|s|)$ on $\Re(s) = 1$, the error term of \eqref{periodic zeta bound} contributes $\ll e^{-At}$. 

Next, using Stirling's formula, we truncate the integral at height $\Im(s)=\pm C\sqrt{t}$ for some sufficiently large $C$ at the cost of an error $\ll e^{-At}$. Then, using Stirling's formula again,
\begin{equation*}
\tau^{-s}\chi(\tfrac12+it-s)=\tau^s\chi(\tfrac12+it)(1+O(\tfrac {1+|s|^2}{t})).
\end{equation*}
The error term here contributes 
\[
\ll t^{-1/2}\sum_{m\geqs 1}\frac{1}{m^{3/2}}\ll t^{-1/2}
\]
and re-extending the integrals also gives an acceptable error. 

Putting this together, we find
\begin{multline*}
\zeta(\tfrac12+it,\alpha)
=
\frac{1}{2\pi i }\int_{1-i\infty}^{1+i\infty}\Big\{\zeta(\tfrac12+it+s,\alpha)+\\\chi(\tfrac12+it)P(\tfrac12 - it + s,-\alpha)\Big\}\tau^s G(s)\frac{ds}{s} + O(t^{-1/2})
\end{multline*}
We can now unfold the Hurwitz zeta function and the periodic zeta function as generalised Dirichlet series and interchange the order of integration and summation to conclude the lemma. 

\end{proof}


Note that the proof above does not require $\ol{G(z)} = G(\ol{z})$. However, making this assumption ensures that the weight $w_t(x)$ is real, which will be used crucially in the sequel. The prototypical choice for the kernel is $G(z) = e^{z^2}$.

By shifting the line of integration in the weights $w_t(x)$ to $\Re(s)=\pm A$ for some $A>1$, we find that 
\begin{equation} \label{weight bound}
w_t(x)
=
\begin{cases}
1+O((x/\sqrt{t})^A) & \text{ if } x\leqs \tau
\\
O((\sqrt{t}/x)^A) & \text{ if } x> \tau.
\end{cases}
\end{equation}
Accordingly, we may restrict the sums in Lemma~\ref{afe lem} to $m,n\leqs T^{1/2+\epsilon}$ if necessary.  In a similar way we find that 
\begin{equation}\label{weight diff}
t^j\frac{d^j}{dt^j}w_t(x)\ll \min(x/\sqrt{t},\sqrt{t}/x)^A.
\end{equation}
These estimates will be used throughout.


\section{Statement of Main Propositions and Proofs of Theorem~\ref{main thm}, Corollary~\ref{main cor}, and Corollary~\ref{main cor 2}}

In the rest of the paper, the majority of work will be in proving the following two fourth moment asymptotics. 

\begin{prop}\label{S1 prop} Let $\Phi(t)$ be a smooth non-negative function of compact support in $[1/2,5/2]$ with derivatives satisfying $\Phi^{(j)}(t)\ll_j T^\epsilon$ for all fixed $j\geqs 0$. Then for large $T$ and irrational $\alpha$ satisfying $\mu(\alpha)<3$, we have 
 \[
 \int_\mathbb{R}\bigg|\sum_{n\geqs 0}\frac{w_t(n+\alpha)}{(n+\alpha)^{1/2+it}}\bigg|^4\Phi(t/T)dt
 =
\frac{\hat{\Phi}(0)}{2}T(\log T)^2+O(T(\log T)^{5/3}).
 \]
 \end{prop}
 \begin{prop}\label{S2 prop} Under the same conditions we have   
 \[
 \int_\mathbb{R}\bigg|\sum_{n\geqs 1}\frac{e(-n\alpha)}{n^{1/2-it}}w_t(n)\bigg|^4\Phi(t/T)dt
 =
\frac{\hat{\Phi}(0)}{2}T(\log T)^2+O(T\log T).
 \]
 \end{prop}

We now show how to deduce our main results from the above propositions.

\begin{proof}[Proof of Theorem~\ref{main thm}]
Take $\Phi$ to  satisfy the conditions of Proposition~\ref{S1 prop} whilst majorising the characteristic function of the interval $[1,2]$ so that 
\[
\int_T^{2T}|\zeta(\tfrac12+it,\alpha)|^4dt\leqs \int_\mb{R}|\zeta(\tfrac12+it,\alpha)|^4\Phi(t/T)dt.
\]
Write
\[
S_1
=
\sum_{n\geqs 0}\frac{1}{(n+\alpha)^{1/2+it}}w_t(n+\alpha),
\qquad
S_2
=
\chi(\tfrac{1}{2}+it)\sum_{n\geqs 1}\frac{e(-n\alpha)}{n^{1/2-it}}w_{t}(n)
\]
so that by Lemma~\ref{afe lem},
\[
\zeta(\tfrac12+it,\alpha)=S_1+S_2+O(t^{-1/2}).
\]
Taking the absolute fourth power and expanding gives 
\begin{multline}\label{decomposition}
\int_\mathbb{R}|\zeta(\tfrac12+it,\alpha)|^4\Phi(t/T)dt
= 
\int_\mathbb{R}\bigg(
|S_1|^4+4|S_1|^2|S_2|^2+|S_2|^4
\\
+4\Re S_1^2\ol{S_1S_2}
+2\Re S_1^2\ol{S_2}^2
+4\Re S_1S_2\ol{S_2}^2
\bigg)\Phi(t/T)dt
+E(T)
\end{multline}
where 
\[
E(T)\ll \int_\mathbb{R}\bigg(
|S_1|^3+|S_2|^3+|S_1||S_2|^2+|S_1|^2|S_2|
\bigg)\Phi(t/T)\frac{dt}{t^{1/2}}.
\]
By  Propositions~\ref{S1 prop} and \ref{S2 prop} along with H\"older's inequality we find 
\begin{equation} \label{wasteful holder}
\int_\mb{R} |S_1|^2|S_2|^2\Phi(t/T)dt\ll T(\log T)^2,\qquad \int_\mb{R} |S_i||S_j|^3\Phi(t/T)dt\ll T(\log T)^2
\end{equation}
and similarly
\begin{equation*}
E(T)\ll \bigg(\int_\R \Phi(t/T)\frac{dt}{t^2}\bigg)^{1/4}\Bigg[\max_{j\in\{1,2\}} \bigg(\int_\R |S_j|^4 \Phi(t/T) dt\bigg)^{3/4} \Bigg] 
\ll T^{1/2} (\log T)^{3/2}.
\end{equation*}
Theorem~\ref{main thm} then follows. 
\end{proof}

\begin{proof}[Proof of Corollary~\ref{main cor}]
For the upper bound, H\"older's inequality gives for any $0\leqs k\leqs 2$,
\[
\int_T^{2T}|\zeta(\tfrac12+it,\alpha)|^{2k}dt\leqs T^{1-k/2}\bigg(\int_T^{2T}|\zeta(\tfrac12+it,\alpha)|^4dt\bigg)^{k/2}\ll T(\log T)^{k}.
\]
By Rane's bound \eqref{rane}, for $k\geqs 1$ we have 
\[
T\log T\sim \int_T^{2T}|\zeta(\tfrac12+it,\alpha)|^{2}dt\leqs T^{1-1/k}\bigg(\int_T^{2T}|\zeta(\tfrac12+it,\alpha)|^{2k}dt\bigg)^{1/k}
\]
whilst for $0\leqs k\leqs 1$ we have 
\begin{align*}
T\log T
\ll &
\bigg(\int_T^{2T}|\zeta(\tfrac12+it,\alpha)|^{2k}dt\bigg)^{\tfrac{1}{2-k}}\bigg(\int_T^{2T}|\zeta(\tfrac12+it,\alpha)|^{4}dt\bigg)^{\tfrac{1-k}{2-k}}
\\
\ll &
\bigg(\int_T^{2T}|\zeta(\tfrac12+it,\alpha)|^{2k}dt\bigg)^{\tfrac{1}{2-k}}(T\log ^2T)^{\tfrac{1-k}{2-k}}
\end{align*}
and rearranging gives the lower bound.
\end{proof}

\begin{proof}[Proof of Corollary~\ref{main cor 2}]
By an argument similar to \eqref{periodic zeta bound} with $s = 0$ along with the bound $P(\tfrac12+it,\alpha)\ll t$, which follows in the usual way by Stieltjes integration, one has that
\[ \zeta(\tfrac12+it,\alpha) = \chi(\tfrac12+it) P(\tfrac12 - it,-\alpha) + O(te^{-\pi t}). \]
Whence,
\begin{equation}\label{h to p}
 \int_{T}^{2T} |P(\tfrac12+it,\alpha)|^{2k} dt =(1+O(e^{-0.04T})) \int_{T}^{2T} |\zeta(\tfrac12+it,\alpha)|^{2k}dt+O(e^{-6kT})
 \end{equation}
since the $\chi$-factor has modulus $1$ and the contribution to the integral from those $t$ for which $\zeta(\tfrac12+it,\alpha)\ll e^{-3.1T}$, say, is clearly negligible. The result follows by Corollary~\ref{main cor}.
\end{proof}


\section{Proof of Proposition~\ref{S1 prop}}\label{sec: S1 prop}

\subsection{Diagonal terms} \label{first diag sec}
Expanding the fourth power and pushing the integral through the sum we find
\begin{multline*}
 \int_\mathbb{R}\bigg|\sum_{n\geqs 0}\frac{w_t(n+\alpha)}{(n+\alpha)^{1/2+it}}\bigg|^4\Phi(t/T)dt
 \\
 =
\sum_{n_j\geqs 0}\frac{1}{\prod_{j=1}^4(n_j+\alpha)^{1/2}}\int_\mathbb{R}\bigg(\frac{(n_1+\alpha)(n_2+\alpha)}{(n_3+\alpha)(n_4+\alpha)}\bigg)^{-it}\Phi(t/T) \wit(\ul{n})dt.
\end{multline*}
where 
\[
\wit(\ul{n})
=
w_t(n_1+\alpha) w_t(n_2+\alpha) {w_t(n_3+\alpha)} {w_t(n_4+\alpha)}
\]
for convenience.

The diagonals terms, i.e. \!\!those for which $(n_1+\alpha)(n_2+\alpha)=(n_3+\alpha)(n_4+\alpha)$, contribute
\[
\mc{I}_D:=\sum_{\substack{(n_1+\alpha)(n_2+\alpha)\\=(n_3+\alpha)(n_4+\alpha)}}
\frac{1}{\prod_{j=1}^4(n_j+\alpha)^{1/2}}
\int_\mathbb{R}\Phi(t/T)\wit(\ul{n})dt
\]
and since for irrational $\alpha$ the only solutions of the given equation are $n_1=n_3, n_2=n_4$ and $n_1=n_4,n_2=n_3$, this is  
\[
\int_\mathbb{R}\bigg(2\bigg(\sum_{n\geqs 0}\frac{w_t(n+\alpha)^2}{(n+\alpha)}\bigg)^2 - \sum_{n\geqs 0}\frac{w_t(n+\alpha)^4}{(n+\alpha)^2}\bigg)\Phi(t/T)dt
\]
by symmetry.

By \eqref{weight bound}, 
\begin{equation*}\begin{split}
\sum_{n \geqs 0} \frac{w_t(n+\alpha)^{2}}{(n+\alpha)} 
= & 
\sum_{0\leqs n \leqs \sqrt{t/2\pi} - \alpha} \frac{1}{(n+\alpha)} + O\bigg(\sum_{0\leqs n \leqs \sqrt{t/2\pi} - \alpha} \frac{(n+\alpha)^{A-1}}{t^{A/2}}\bigg) 
\\
& \qquad\qquad\qquad + O\bigg(\sum_{n \geqs \sqrt{t/2\pi} - \alpha} \frac{t^{A/2}}{(n+\alpha)^{A+1}}\bigg) 
\\
= &\frac{1}{2} \log(t/2\pi) + O(1).\end{split}\end{equation*}
Similarly, $ \sum_{n \geqs 0} {w_t(n+\alpha)^{4}}{(n+\alpha)^{-2}} \ll 1  $ and so
\begin{equation*}\mc{I}_D= \frac{1}{2} \int_\R \big(\log(t/2\pi)+O(1)\big)^2 \Phi(t/T) dt = \frac{\hat{\Phi}(0)}{2} T(\log T)^2 + O(T\log T). \end{equation*}
It remains to show that the off-diagonal terms contribute at most $\ll T (\log T)^{5/3}$.

\subsection{Off diagonal terms: initial cleaning}

In this subsection we perform some fairly standard procedures; first restricting the off-diagonals to the ``close" off-diagonals and then applying Taylor expansions of the summands. 

The off diagonal terms are given by 
\[
\mc{I}_O
:=
\sum_{\substack{n_j \geqs 0\\(n_1+\alpha)(n_2+\alpha)\\\neq (n_3+\alpha)(n_4+\alpha)}}\frac{1}{\prod_{j=1}^4(n_j+\alpha)^{1/2}}\int_\mathbb{R}\bigg(\frac{(n_1+\alpha)(n_2+\alpha)}{(n_3+\alpha)(n_4+\alpha)}\bigg)^{-it}\Phi(t/T)\wit(\ul{n})dt.
\]
We divide this sum into three cases, depending on whether $\max(n_1,n_2)$ is greater than, less than, or equal to $\max(n_3,n_4)$.

 The lattermost case contributes a lower order term; to see this, by symmetry it suffices to bound
\begin{equation*} \sum_{\substack{0\leqs n_2, n_4 \leqs n_1\\n_2 \neq n_4}} \frac{1}{(n_1+\alpha)(n_2 + \alpha)^{1/2}(n_4 + \alpha)^{1/2}} \int_\R \bigg(1 + \frac{h}{n_4+\alpha}\bigg)^{-it}\Phi(t/T) \wit(\ul{n}) \,dt, \end{equation*}
where $h = n_2 - n_4$ and where $\wit(\ul{n}) = w_t(n_1+\alpha)^2 w_t(n_2 + \alpha) w_t(n_4 + \alpha)$. 
 Integrating by parts $j$ times using \eqref{weight diff}, we find
\begin{equation*} \int_\R \bigg(1 + \frac{h}{n_4 + \alpha}\bigg)^{-it} \Phi(t/T) \wit(\ul{n}) \, dt \ll_j \frac{T}{|T\log(1 + \tfrac{h}{n_4 + \alpha})|^j}. \end{equation*}
Thus, if $h/(n_4 + \alpha) \gg T^{-1+\epsilon}$, then the above quantity is $\ll T^{-A}$ on taking $j$ large enough and hence such terms may be omitted. For the remaining terms, without losing generality we only bound the terms with $h > 0$. The close diagonal condition $h \ll (n_4 + \alpha) T^{-1 + \epsilon}$ implies that $n_2 + \alpha = (n_4 + \alpha)(1 + O(T^{-1+\epsilon}))$. Recalling that the weights restrict to $n_j \leqs T^{1/2+\epsilon}$ and estimating the integral trivially, we get that this sum is
\begin{equation*} \ll T \sum_{\substack{0 \leqs n_4 \leqs n_1 \leqs T^{1/2 + \epsilon}\\0< h \ll \frac{(n_4 + \alpha)}{T^{1-\epsilon}}}} \frac{1}{(n_1 + \alpha)(n_4 + \alpha)} \ll T^{1/2 + \epsilon}. \end{equation*}
We remark here that a similar argument shows that terms with $n_j = 0$ can be safely ignored, and so we can assume $n_j \geqs 1$.

Now, returning to $\mc{I}_O$, the terms with $\max(n_1,n_2) < \max(n_3,n_4)$ are seen to be conjugates of those with $\max(n_1,n_2) > \max(n_3,n_4)$ due to the symmetry $(n_1,n_2) \leftrightarrow (n_3,n_4)$. Pairing these terms together, we get

\begin{multline*} \Re\bigg(\sum_{\substack{n_j \geqs 1\\ \max(n_3,n_4)>\max(n_1,n_2)\\(n_1+\alpha)(n_2+\alpha)\neq (n_3+\alpha)(n_4+\alpha)}}\frac{1}{\prod_{j=1}^4(n_j+\alpha)^{1/2}}\\ \times \int_\mathbb{R}\bigg(\frac{(n_1+\alpha)(n_2+\alpha)}{(n_3+\alpha)(n_4+\alpha)}\bigg)^{-it}\Phi(t/T)\wit(\ul{n})dt\bigg). \end{multline*}

We now let
\[
(n_1+\alpha)(n_2+\alpha)-(n_3+\alpha)(n_4+\alpha)=h_1+h_2\alpha
\]
with 
\begin{equation}\label{h_i}
h_1=n_1n_2-n_3n_4,\qquad h_2=n_1+n_2-n_3-n_4
\end{equation}
and note that $h_1\ll T^{1+\epsilon}$ and $h_2\ll T^{1/2+\epsilon}$.

We restrict again to close off-diagonals. Integrating by parts $j$ times using \eqref{weight diff} we see that
\[
\int_\mathbb{R} \exp\bigg(it\log(1+\tfrac{h_1+h_2\alpha}{(n_3+\alpha)(n_4+\alpha)})\bigg)\Phi(t/T)\wit(\ul{n})dt\ll_j \frac{T}{|T\log(1+\tfrac{h_1+h_2\alpha}{(n_3+\alpha)(n_4+\alpha)})|^j}.
\]
Thus, if $(h_1+h_2\alpha)/(n_3+\alpha)(n_4+\alpha)\gg T^{-1+\epsilon}$ then the above quantity is $\ll T^{-A}$ on taking $j$ large enough, and so we may omit such terms. For the remaining terms we apply the expansions 
\begin{align*}
\log\bigg(\frac{(n_1+\alpha)(n_2+\alpha)}{(n_3+\alpha)(n_4+\alpha)}\bigg)= &\frac{h_1+h_2\alpha}{(n_3+\alpha)(n_4+\alpha)}+O(T^{-2+\epsilon})
\\
\bigg(\frac{(n_1+\alpha)(n_2+\alpha)}{(n_3+\alpha)(n_4+\alpha)}\bigg)^{-it}= &\exp\bigg(-it\frac{h_1+h_2\alpha}{(n_3+\alpha)(n_4+\alpha)}\bigg)+O(T^{-1+\epsilon})
\\
(n_1+\alpha)(n_2+\alpha)= & (n_3+\alpha)(n_4+\alpha)(1+O(T^{-1+\epsilon})).
\end{align*}

The errors acquired from these approximations are
\begin{equation*} 
 \ll
  \sum_{\substack{n_j \geqs 1\\ \max(n_3,n_4) > \max(n_1,n_2)\\0< |h_1+h_2\alpha|\ll \frac{(n_3+\alpha)(n_4+\alpha)}{T^{1-\epsilon}}}}\frac{1}{(n_3+\alpha)(n_4+\alpha)}\int_\mathbb{R}\frac{\Phi(t/T)|\wit(\ul{n})|}{T^{1-\epsilon}}dt.
\end{equation*}
Pushing the sum inside, recalling by \eqref{weight bound} that $w_t$ restricts the sums to $n_j \leqs T^{1/2+\epsilon}$ and computing the integral, this is 
\begin{equation*} 
 \ll 
 T^\epsilon \sum_{\substack{1\leqs n_j\leqs T^{1/2+\epsilon}\\0< |h_1+h_2\alpha|\ll \frac{(n_3+\alpha)(n_4+\alpha)}{T^{1-\epsilon}}\\ h_1=n_1n_2-n_3n_4\\ h_2=n_1+n_2-n_3-n_4}}\frac{1}{(n_3+\alpha)(n_4+\alpha)}.\end{equation*}
We allow $h_1, h_2, n_3, n_4$ to vary freely. There are $\ll T^{1/2 + \epsilon}$ choices for $h_2$. After fixing that, there are $\ll (n_3+\alpha)(n_4+\alpha)/T^{1-\epsilon}$ choices for $h_1$. Then, there are $\ll T^{1/2+\epsilon}$ choices for $n_3, n_4$. Finally, applying a divisor bound on the constraint $n_1n_2 = h_1 + n_3 n_4$ gives that there are $\ll T^{\epsilon}$ choices for $n_1$ and $n_2$. Note that it is not possible that $h_1 + n_3 n_4 = 0$, since we have discarded the terms with $n_1=0$ or $n_2 = 0$. We see thus that the above sum is $\ll T^{1/2+\epsilon}$, which is good enough. 


Thus far we have 
\begin{multline*}
\mc{I}_O
=
\Re\bigg(\sum_{\substack{n_j \geqs 1\\ \max(n_3,n_4) > \max(n_1,n_2)\\0< |h_1+h_2\alpha|\ll \frac{(n_3+\alpha)(n_4+\alpha)}{T^{1-\epsilon}}}}
\frac{1}{(n_3+\alpha)(n_4+\alpha)}
\\
\times
\int_\mathbb{R}\exp\bigg(it\frac{h_1+h_2\alpha}{(n_3+\alpha)(n_4+\alpha)}\bigg)\Phi(t/T)w_t(\ul{n})dt\bigg)
+O(T^{1/2+\epsilon})
\end{multline*}
Integrating by parts, the main term here is the negative of the imaginary part of

\begin{equation*} \sum_{\substack{n_j \geqs 1\\ \max(n_3,n_4) > \max(n_1,n_2)\\0< |h_1+h_2\alpha|\ll \frac{(n_3+\alpha)(n_4+\alpha)}{T^{1-\epsilon}}}}\frac{1}{h_1 + h_2\alpha} \int_\R \exp\bigg(it\frac{h_1 + h_2\alpha}{(n_3+\alpha)(n_4+\alpha)}\bigg) \frac{d}{dt} (\Phi(t/T) \wit(\ul{n})) \,dt. \end{equation*}
Note that $ \int_\R \frac{d}{dt} (\Phi(t/T) \wit(\ul{n})) \,dt = 0$ due to the compact support of $\Phi$. Subtracting this from inner integral and pushing the sum through, we now focus on
\begin{equation*} \mc{I}' := \sum_{\substack{n_j \geqs 1\\ \max(n_3,n_4) > \max(n_1,n_2)\\0< |h_1+h_2\alpha|\ll \frac{(n_3+\alpha)(n_4+\alpha)}{T^{1-\epsilon}}}}\frac{1}{h_1 + h_2\alpha}\bigg\{\exp\bigg(it\frac{h_1 + h_2\alpha}{(n_3+\alpha)(n_4+\alpha)}\bigg) - 1\bigg\} \wit(\ul{n}).
\end{equation*}
The purpose of this subtraction is to emulate the unsmoothed integral $\int_T^{2T} \exp(iHt)dt=(e^{2iHT}-e^{iHT})/iH$ which more visibly has a finite limit when $H\to0$. This is required since $h_1+h_2\alpha$ can become arbitrarily small (depending on Diophantine properties of $\alpha$).

We will show in the next subsection that for $t \asymp T$ and $\mu(\alpha) < 3$, $\mc{I'} \ll T(\log T)^{5/3}$. Replacing $\wit(\ul{n})$ here by $\frac{d}{dt} w_t(\ul{n})$ can be dealt in a similar way using \eqref{weight diff} with the resultant bound having a factor of $T^{-1}$ compared with the above bound for $\mc{I}'$.
 
From these we acquire 
\begin{equation*} \mc{I}_{O} \ll \int_\R \bigg(\frac{1}{T} |\Phi'(t/T)|\cdot T(\log T)^{5/3}+|\Phi(t/T)|(\log T)^{5/3}\bigg) \,dt \ll T(\log T)^{5/3}, \end{equation*}
as desired to finish the proof of the proposition. 

We record at this point the trivial observation
\begin{equation}\label{exp bound} \exp(itx) - 1 \ll \min(1,T|x|), \end{equation}
for $t \asymp T$ and uniformly in $x$ which shall be used in the sequel.

\subsection{Parametrising solutions to the off-diagonal equations}

It remains to bound $\mc{I}'$. Since $\mc{I}'$ is invariant under swapping $n_3 \leftrightarrow n_4$, without loss of generality we can add the hypothesis $n_4 \leqs n_3$ to the sum. Thus, we have that $n_1, n_2 < n_3$. We now find parametrised solutions to \eqref{h_i} and express our sum over $n_j$ in terms of these parameters.  

Fix $h_1, h_2$, and $n_3$ and set $k = n_3 - n_1$. We get from \eqref{h_i} that
\begin{equation} \label{s1n2n4} n_2 = h_2 + n_4 + k, \qquad n_4 = n_3 - h_2 - k + \frac{n_3h_2 -h_1}{k}. \end{equation}
This tells us that $k \mid h_2 n_3 - h_1$ and $k \geqs 1$.  Relabelling $n_3 \mapsto n$; rearranging the sum with $h_1$, $h_2$, $n$ outside and $k$ inside; and setting 
\[
n^* = n - k - h_2 + (n h_2 - h_1)/k,
\] 
we find that it suffices to bound
\begin{equation*} \sum_{\substack{n\geqs 1\\0 < |h_1 + h_2\alpha| \ll \frac{(n+\alpha)^2}{T^{1-\epsilon}}}} \sum_{\substack{k \mid n h_2 - h_1}}\frac{1}{h_1 + h_2\alpha}\bigg\{\exp\bigg(it\frac{h_1 + h_2\alpha}{(n+\alpha)(n^*+\alpha)}\bigg) - 1\bigg\} \WIt(h_1,h_2,n,k), \end{equation*}
where, for brevity, 
\begin{multline} \label{WIt} \WIt(h_1,h_2,n,k) = w_t(n+k+\alpha)w_t(n+\tfrac{nh_2-h_1}{k}+\alpha)\\\times w_t(n+\alpha)w_t(n-k-h_2 + \tfrac{nh_2-h_1}{k}+\alpha). \end{multline}
 
We rearrange the innermost sum in terms of the greatest common divisor $g = (k,h_2)$. Clearly, the divisibility constraint $k \mid nh_2 - h_1$ ensures that $g \mid h_1$ as well. Replacing $(k,h_1,h_2)$ with $(gk, gh_1,gh_2)$ and pulling the sum over $g$ out, we get 
\begin{multline*} 
\sum_{1 \leqs g\ll T^{\epsilon}}\frac{1}{g} \sum_{\substack{n\geqs 1\\0 < |h_1 + h_2\alpha| \ll \frac{(n+\alpha)^2}{g T^{1-\epsilon}}}} 
\sum_{\substack{k \mid nh_2 - h_1\\(k,h_2) = 1}}
\frac{1}{h_1 + h_2\alpha}\\\times \bigg\{\exp\bigg(igt\frac{h_1 + h_2\alpha}{(n+\alpha)(n_g^*+\alpha)}\bigg)-1\bigg\} \WIt(gh_1,gh_2,n,gk)
\end{multline*}
where 
\begin{equation}\label{n^*g}
n^*_g= n - gk - gh_2 + (n h_2 - h_1)/k.
\end{equation}

We divide the above sum into three cases:
\begin{itemize}

\item $h_1$ and $h_2$ have the same sign.

\item $h_1$ and $h_2$ have opposite signs and $|h_1 + h_2\alpha| \geqs \frac{1}{2g}$.

\item $h_1$ and $h_2$ have opposite signs and $|h_1 + h_2\alpha| < \frac{1}{2g}$.

\end{itemize}

\subsection{Denominator case I: same sign} \label{S1 same signs}
We deal with $h_1, h_2 \geqs 0$, noting that a similar argument works for $h_1,h_2 \leqs 0$. The constraint $|h_1 + h_2\alpha| \ll (n+\alpha)^2 g^{-1} T^{-1+\epsilon}$ now implies that in fact $h_1,h_2 \ll g^{-1}T^\epsilon$. Thus by \eqref{exp bound}, such terms trivially contribute

\begin{equation*}\begin{split}& \ll \sum_{1\leqs g \ll T^\epsilon}\frac{1}{g}\sum_{h_1,h_2 \ll \frac{T^\epsilon}{g}} \frac{1}{h_1 + h_2\alpha} \sum_{n \geqs 1} \sum_{\substack{k\mid nh_2 - h_1\\(k,h_2) = 1}} |\WIt(gh_1,gh_2,n,gk)| 
 \\ & 
 \ll\sum_{1\leqs g \ll T^\epsilon}\frac{1}{g}\sum_{h_1,h_2 \ll \frac{T^\epsilon}{g}} \frac{1}{h_1 + h_2\alpha}\sum_{\substack{k\ll T^{1/2+\epsilon}\\(k,h_2) = 1}} \sum_{n\equiv \ol{h_2}h_1 \!\!\!\mod k}  |\WIt(gh_1,gh_2,n,gk)| 
  \\ & 
  \ll\sum_{1\leqs g \ll T^\epsilon}\frac{1}{g}\sum_{h_1,h_2 \ll \frac{T^\epsilon}{g}} \frac{1}{h_1 + h_2\alpha}\sum_{\substack{k\ll T^{1/2+\epsilon}\\(k,h_2) = 1}} \frac{T^{1/2+\epsilon}}{k} \ll T^{1/2+\epsilon}, \end{split}\end{equation*}
which is good enough. Here we have used $1 \leqs k \leqs n \ll T^{1/2+\epsilon}$, and the estimate $\WIt \ll \1_{n\leqs T^{1/2+\epsilon}}$. 

\subsection{Denominator case II: opposite signs, large denominator} \label{S2 opp signs, large denominator}
We now turn to $h_1 \geqs 0$, $h_2 \leqs 0$ and $|h_1 + h_2\alpha|\geqs 1/(2g)$. As before, the argument for $h_1 \leqs 0$, $h_2 \geqs 0$ and $|h_1 + h_2\alpha| \geqs 1/(2g)$ is similar. Thus, relabelling $h_2 \mapsto -h_2$, and using \eqref{exp bound}, we now have both $h_1,h_2\geqs 0$, and the sum to control is

\begin{equation*} \sum_{1\leqs g\ll T^\epsilon}\frac{1}{g}\sum_{\substack{h_1,h_2\geqs 0\\\frac{1}{2g} < |h_1 - h_2 \alpha| \ll \frac{T^{\epsilon}}{g}}}\frac{1}{|h_1 - h_2\alpha|}\sum_{n\geqs 1}\sum_{\substack{k \mid nh_2 + h_1\\(k,h_2) = 1}}|\WIt(gh_1,-gh_2,n,gk)| 
. \end{equation*} 
We can quickly exclude the case $h_2=0$ since this contributes $\ll T^{1/2+\epsilon}$ on applying the divisor bound $d(h_1)\ll T^{\epsilon}$ and applying, say, \eqref{WIt bound} below for the sum over $n$. 

Recall that in our previous notation $n_1 = n - gk \geqs 0$, and hence $k \leqs n/g$. Further, by \eqref{s1n2n4},
\begin{equation*} 0 \leqs n_2 = n\Big(1 -\frac{gh_2}{gk}\Big) - \frac{gh_1}{gk} \leqs n\Big(1 - \frac{h_2}{k}\Big), \end{equation*}
whence $h_2\leqs k$. Thus, swapping the order of summation, 
\begin{equation*} \sum_{ g\ll T^\epsilon}\frac{1}{g}\sum_{\substack{h_1\geqs 0,h_2\geqs 1\\\frac{1}{2g} < |h_1 - h_2 \alpha| \ll \frac{T^{\epsilon}}{g}}}\frac{1}{|h_1 - h_2\alpha|}\sum_{\substack{k\geqs h_2\\(k,h_2) = 1}}\sum_{\substack{n\equiv -\ol{h_2}h_1(k)\\n\geqs gk}}|\WIt(gh_1,-gh_2,n,gk)| . \end{equation*} 

Now, using \eqref{weight bound} and recalling the definition of $\WIt$ in \eqref{WIt},
\begin{equation}\label{WIt bound}
\WIt(gh_1,-gh _2,n,gk) 
\ll 
w_t(n+\alpha) \ll \1_{n+\alpha\leqs \tau} + \1_{n+\alpha>\tau}(\tfrac{\tau}{n+\alpha})^3.
\end{equation}
Hence,
\begin{equation*} \sum_{\substack{n\equiv -\ol{h_2}h_1\!\!\!\mod k\\n\geqs gk}} |\WIt| \ll \1_{k\leqs \tau/g}(\tfrac{\tau}{k}) +  \1_{k > \tau/g}(\tfrac{\tau^3}{g^2k^3}) \end{equation*}
where in the sum over $n+\alpha>\tau$ we have written $1=\mathds{1}_{k\leqs \tau/g}+\mathds{1}_{k> \tau/g}$ and for the first term summed over $n>\tau-\alpha$ and for the second summed over $n\geqs gk$. 
Performing the sum over $k$ in a similar fashion, we are left with bounding
\begin{equation*} \tau \sum_{ g\ll T^\epsilon} \frac{1}{g} \sum_{\substack{h_1\geqs 0,h_2\geqs 1\\\frac{1}{2g} < |h_1 - h_2\alpha| \ll \frac{T^\epsilon}{g}}} \frac{1}{|h_1 - h_2\alpha|}\bigg(\1_{h_2\leqs \tau/g}(\log(\tfrac{\tau}{gh_2})+1)+\1_{h_2>\tau/g}(\tfrac{\tau}{gh_2})^2\bigg).\end{equation*}
Pulling the sum over $h_2$ out this is 
\begin{equation*} \begin{split} 
&\,\,\tau
\sum_{g\leqs T^{\epsilon}}\frac{1}{g}
\sum_{h_2\geqs 1}\bigg(\1_{h_2\leqs\tau/g}(\log(\tfrac{\tau}{gh_2})+1)+\1_{h_2>\tau/g}(\tfrac{\tau}{gh_2})^2\bigg)
\sum_{\substack{h_1\geqs 0 \\\frac{1}{2g}<|h_1 - h_2\alpha|\ll \frac{T^\epsilon}{g}}} \frac{1}{|h_1 - h_2\alpha|}
\\ 
\ll & 
\,\,\tau
\sum_{g\leqs T^{\epsilon}}\frac{1}{g}
\sum_{h_2\geqs 1}\bigg(\1_{h_2\leqs\tau/g}(\log(\tfrac{\tau}{gh_2})+1)+\1_{h_2>\tau/g}(\tfrac{\tau}{gh_2})^2\bigg)
\big(g+\log T\big)\\ 
\ll & 
\,\,\tau
\sum_{g\leqs T^{\epsilon}}\frac{1}{g}\cdot\frac{\tau}{g}\cdot(g+\log T)
\ll
\tau^2\log T 
\ll 
T\log T 
\end{split} \end{equation*}
since $\sum_{m\leqs x}\log(x/m)\ll x$.
\subsection{Denominator case III: opposite signs, small denominator}

 It is here that our assumptions on the Diophantine properties of $\alpha$ begin to play a role in the moments. We shall use Lemma~\ref{kruse} to control these sums.

We consider $|h_1 + h_2\alpha| <1/(2g)$ and again restrict our attention to $h_1 \geqs 0$, $h_2 \leqs 0$ relabelling $h_2 \mapsto -h_2$. Recall that $k \geqs 1$ so that $g = (k,h_2) \geqs 1$ and hence $|h_1 - h_2\alpha| = \| h_2\alpha\|$ arising from the unique choice $h_1=\lfloor h_2\alpha\rceil$ - the closest integer to $h_2\alpha$. Thus as $h_2$ varies $h_1$ is fixed. Renaming $h_2$ as $h$  it suffices to show that 
\begin{multline*}\label{sd bound} \sum_{1 \leqs g\ll T^{\epsilon}}\frac{1}{g} \sum_{\substack{n,h \geqs 1\\\normh <\frac{1}{2g}}} \sum_{\substack{k \mid nh + \near{h} \\(k,h) = 1}}\frac{1}{\normh}\\\times \bigg\{\exp\bigg(igt\frac{\normh}{(n+\alpha)(n_{g,-}^*+\alpha)}\bigg)-1\bigg\} \WIt'(g,h,n,k)
\\
\ll
T(\log T)^{5/3},
 \end{multline*}
where 
\begin{equation*} \WIt'(g,h,n,k) = \WIt(g\lfloor h\alpha\rceil,-gh,n,gk) \end{equation*}
and, after transforming \eqref{n^*g},
\[
n^*_{g,-}= n - g(k - h) - (n h + \lfloor h\alpha\rceil)/k.
\]

As argued before, $h \leqs k \leqs n/g$. Further, note that $h = k$ if and only if $n_2 = 0$, and hence as noted in the beginning of this section, such terms can be ignored. The same applies to the case $n=gk$ since we had $n_1 = n - gk$. Swapping the order of summation, 
\begin{multline*} \sum_{1 \leqs g\ll T^{\epsilon}}\frac{1}{g} \sum_{\substack{h\geqs 1\\\normh<\frac{1}{2g}}}\frac{1}{\normh}\sum_{\substack{k> h\\(h,k)=1}} \sum_{\substack{n\equiv -\ol{h}\near{h} \!\!\!\mod k\\n> gk}}\bigg\{\exp\bigg(igt\frac{\normh}{(n+\alpha)(n^*_{g,-}+\alpha)}\bigg)-1\bigg\} \\\times \WIt'(g,h,n,k), \end{multline*}
Guided by \eqref{exp bound} and $n^*_{g,-}\leqs n$, we divide the innermost sum into the cases $n \leqs  (gt\normh)^{1/2}(\log t)^{1/3}$ and $n > (gt\normh)^{1/2}(\log t)^{1/3}$ and call the resulting sums $\mc{I}_1$ and $\mc{I}_2$ respectively. 

Using $h\ll T^{1/2+\epsilon}$ together with \eqref{exp bound} we get
\begin{equation*} \mc{I}_1 \ll \sum_{1 \leqs g\ll T^{\epsilon}}\frac{1}{g} \sum_{\substack{1 \leqs h\ll T^{1/2+\epsilon}\\\normh<\frac{1}{2g}}} \frac{1}{\normh}\sum_{\substack{k> h\\(k,h)=1}} \sum_{\substack{n\equiv -\ol{h}\near{h} \!\!\!\mod k\\gk< n \leqs  (gt\normh)^{1/2}(\log t)^{1/3}}}1. \end{equation*}
Note that the innermost sum is empty unless $k\leqs(t\norm{h\alpha}/g)^{1/2}(\log t)^{1/3}$, whence, performing the sum over $n$, 
\begin{equation*}
\mc{I}_1 
	\ll T^{1/2} (\log T)^{1/3} \sum_{1 \leqs g\ll T^{\epsilon}}\frac{1}{\sqrt{g}} \sum_{\substack{1 \leqs h\ll T^{1/2+\epsilon}\\\normh<\frac{1}{2g}}} \frac{1}{\normh^{1/2}} \sum_{\substack{h< k\leqs (t\norm{h\alpha}/g)^{1/2}(\log t)^{1/3}}} \frac{1}{k}. \end{equation*}
Again, the innermost sum is empty unless $h< (t\norm{h\alpha}/g)^{1/2}(\log t)^{1/3}<t^{1/2}(\log t)^{1/3}/g^{3/2}$, and thus,
\begin{align*} 
\mc{I}_1 
\ll &
\,\,T^{1/2}(\log T)^{4/3}\sum_{1 \leqs g\ll T^{\epsilon}}\frac{1}{\sqrt{g}}
 \sum_{1\leqs h\leqs t^{1/2}(\log t)^{1/3}/g^{3/2}} \frac{1}{\normh^{1/2}} 
 \ll 
\,\, T(\log T)^{5/3}
\end{align*}
by Lemma~\ref{kruse}. 

For $\mc{I}_2$, we use the other half of \eqref{exp bound} to find 
\begin{equation*}
\mc{I}_2\ll T\sum_{1 \leqs g\ll T^{\epsilon}}\sum_{\substack{h\geqs 1\\\normh<\frac{1}{2g}}}\sum_{\substack{k> h\\(h,k)=1}}
 {\sum_n}^\star \frac{1}{(n+\alpha)(n^*_{g,-}+\alpha)} \WIt'(g,h,n,k), \end{equation*} 
where $\star$ denotes the constraints \[n\equiv -\ol{h}\near{h}\!\!\!\mod k,\] \[ \max(gk,(gt\normh)^{1/2}(\log t)^{1/3})< n \leqs T^{1/2+\epsilon}.\]

Now, since $n^*_{g,-} = n - g(k - h) - (nh+\near{h})/k$ we have 
\begin{equation*}\label{n* bound}
\begin{split}
n_{g,-}^*+\alpha
= &
(1-\tfrac{h}{k})(n+\alpha)-g(k-h)-(\near{h}-h\alpha)/k
\\
\geqs &
\tfrac{k-h}{k}(n+\alpha-gk)-\tfrac{1}{2k}
\geqs 
\tfrac{k-h}{2k}(n+\alpha-gk)
\end{split}
\end{equation*}
since $k-h\geqs 1$ and $n+\alpha-gk\geqs 1+\alpha\geqs 1$. Also, $\star$ implies
\[ \frac{1}{n+\alpha} \ll \frac{1}{\sqrt{g}} \cdot \frac{1}{T^{1/2}(\log T)^{1/3}}\cdot \frac{1}{\normh^{1/2}}.\]
Inserting these bounds above, we get
\begin{equation*}
\mc{I}_2\ll \frac{T^{1/2}}{(\log T)^{1/3}}\sum_{1 \leqs g\ll T^{\epsilon}}\frac{1}{\sqrt{g}}\sum_{\substack{h\geqs 1\\\normh<\frac{1}{2g}}}\frac{1}{\normh^{1/2}}\sum_{\substack{k> h\\(h,k)=1}}\frac{k}{k-h}
 {\sum_n}^\star \frac{1}{(n-gk)} \WIt'(g,h,n,k), \end{equation*}

By \eqref{WIt bound} we have 
\begin{equation*} \label{wtpp} \WIt' \ll \1_{n+\alpha\leqs \tau} + \1_{n+\alpha>\tau}(\tfrac{\tau}{n+\alpha})^3.\end{equation*}
Thus,
\begin{equation*}\begin{split} {\sum_n}^\star \frac{1}{(n-gk)} \WIt' & =  \bigg[\,\,\,\sideset{}{^\star}\sum_{gk < n \leqs 2gk} + \sideset{}{^\star}\sum_{n > 2gk}\bigg] \frac{1}{(n-gk)} \WIt' \\& \ll \frac{\log T}{k}\bigg(\1_{k \leqs \tau/g} + \1_{k > \tau/g} (\tfrac{\tau}{g k})^3\bigg) \end{split}\end{equation*}
where in the first sum here, we have used $(n+\alpha) \asymp gk$, while in the second we have computed the logarithmic sum when $n+\alpha\leqs \tau$ and for $n+\alpha>\tau$ trivially bounded $(n-gk)^{-1} \leqs (2n)^{-1}$ and applied $1=\mathds{1}_{k\leqs \tau/g}+\mathds{1}_{k> \tau/g}$, summing over $n\gg \tau$ in the first term and $n\gg gk$ in the second. The log-factors arising in this process can always be replaced by $\log T$, since we have $gk \leqs T^{1/2+\epsilon}$ as entailed by the weights.

The sum over $k$ can be performed similarly, to get
\begin{equation*} \sum_{\substack{k>h\\ (k,h)=1}}\frac{1}{k-h}\bigg(\1_{k \leqs \tau/g} + \1_{k > \tau/g} (\tfrac{\tau}{g k})^3\bigg) \ll \log T\bigg(\1_{h\leqs \tau/g} + \1_{h>\tau/g} (\tfrac{\tau}{gh})^3\bigg). \end{equation*}
Thus,
\begin{equation*} \begin{split}
\mc{I}_2 
& \ll T^{1/2}(\log T)^{5/3} \sum_{1\leqs g\ll T^\epsilon} \frac{1}{\sqrt{g}}\sum_{\substack{h\geqs 1\\\normh<\frac{1}{2g}}}\frac{1}{\normh^{1/2}}\bigg(\1_{h\leqs \tau/g} + \1_{h>\tau/g} (\tfrac{\tau}{gh})^3\bigg) 
\\ 
& \ll T (\log T)^{5/3}, 
\end{split} \end{equation*}
by partial summation and Lemma~\ref{kruse}.

\section{Proof of Proposition~\ref{S2 prop}}\label{sec: S2 prop}
\subsection{Diagonal terms} Let 
\[
\mc{J}
=
 \int_\mathbb{R}\bigg|\sum_{n\geqs 1}\frac{e(-n\alpha)}{n^{1/2-it}}\wjt(n)\bigg|^4\Phi(t/T)dt.
\]
Expanding the fourth power we find 
\begin{align*}
\mc{J}
 = &
 \sum_{n_j\geqs 1} \frac{e(-\alpha(n_1+n_2-n_3-n_4))}{\sqrt{n_1n_2n_3n_4}}
 \int_\mb{R}\bigg(\frac{n_1n_2}{n_3n_4}\bigg)^{it}\wjt(\ul{n})\Phi(t/T)dt
 \\
 = &
 \mc{J}_D+\mc{J}_O
\end{align*}
where $\wjt(\ul{n})=\prod_{j=1}^4 w_t(n_j)$, $\mc{J}_D$ is the sum with $n_1n_2=n_3n_4$ and $\mc{J}_O$ is the remaining off-diagonal sum.  $\mc{J}_D$ can be further decomposed into the sum for which $n_1+n_2=n_3+n_4$, and that for which $n_1+n_2\neq n_3+n_4$. In the first case $\{n_1,n_2\}$ is a permutation of $\{n_3,n_4\}$ and thus we find 
\begin{multline*}
\mc{J}_D
=
\int_\mb{R}\bigg(2\bigg(\sum_{n\geqs 1}\frac{w_t(n)^2}{n}\bigg)^2-\sum_{n\geqs 1}\frac{w_t(n)^4}{n^2}\bigg)\Phi(t/T)dt
\\ +
\sum_{\substack{n_1n_2=n_3n_4\\n_1+n_2\neq n_3+n_4}}
\frac{e(-\alpha(n_1+n_2-n_3-n_4))}{\sqrt{n_1n_2n_3n_4}}
 \int_\mb{R}\wjt(\ul{n})\Phi(t/T)dt.
\end{multline*}

By a calculation similar to that in \S\ref{first diag sec}, the first term on the right is
\begin{align*}
2\int_\mb{R}\bigg(\sum_{n\geqs 1}\frac{w_t(n)^2}{n}\bigg)^2\Phi(t/T)dt
= &
2\int_\mb{R}\bigg(\frac{1}{2}\log(t/2\pi )+O(1)\bigg)^2\Phi(t/T)dt
\\
= &
\frac{\hat{\Phi}(0)}{2}T(\log T)^2+O(T\log T)
\end{align*}
whilst the second term is clearly $O(T)$. The third term is quite delicate and here we have the following lemma. 

\begin{lem}Suppose $\alpha$ is an irrational with $\mu(\alpha)<3$. Then 
\[
\sum_{\substack{n_1n_2=n_3n_4\\n_1+n_2\neq n_3+n_4}}
\frac{e(-\alpha(n_1+n_2-n_3-n_4))}{\sqrt{n_1n_2n_3n_4}} \int_\mb{R}\wjt(\ul{n})\Phi(t/T)dt
 =
c\hat{\Phi}(0)T+O(T/(\log T)^\epsilon)
\]
where 
\[
c=c_T
=
\sum_{d\leqs (\log T)^{4+\epsilon}}\frac{\mu(d)}{d^2}\sum_{m,\ell\geqs 1}\frac{1}{m\ell}\int_1^\infty\int_1^\infty\frac{1}{(m+y)^{2}(\ell+x)^{2}}\Big(\sum_{h\leqs y}\sum_{k\leqs x}e(-\alpha d hk) \Big)dxdy
\]
which satisfies $c_T\ll \log\log T$.
\end{lem}
\begin{proof}


We push the sum through the integral and reparametrize the sum. 

 Let $m_1 = (n_1,n_3)$ and $m_2 = (n_2,n_4)$. Then $m_1m_2$ divides both sides of $n_1n_2 = n_3n_4$. Dividing throughout, we see that it must be the case that
\begin{equation*} \frac{n_1}{(n_1,n_3)} = \frac{n_4}{(n_2,n_4)},\qquad \frac{n_2}{(n_2,n_4)} = \frac{n_3}{(n_1,n_3)}.\end{equation*}
Calling the former $\ell_1$ and the latter $\ell_2$, we see that the solutions to $n_1n_2 = n_3n_4$ are uniquely parametrized by 
\begin{equation*}\begin{split}
n_1 = m_1\ell_1, \qquad \qquad  n_2 = m_2\ell_2, \\ n_3 = m_1\ell_2, \qquad \qquad n_4 = m_2\ell_1, \end{split} \end{equation*}
with $(\ell_1,\ell_2) = 1$. Furthermore, $n_1 + n_2 - n_3 - n_4 = (m_1 - m_2)(\ell_1 - \ell_2)$, and so the constraint $n_1 + n_2 \neq n_3 + n_4$ is equivalent to $m_1 \neq m_2$ and $\ell_1 \neq \ell_2$. The sum then becomes,
\begin{equation*}
\sum_{\substack{m_j,\ell_j\geqs 1\\m_1 \neq m_2,\ell_1\neq\ell_2\\(\ell_1,\ell_2) = 1}}\frac{e(-\alpha(m_1 - m_2)(\ell_1-\ell_2))}{m_1m_2\ell_1\ell_2}
\WJt(m_1,m_2,\ell_1,\ell_2)
\end{equation*}
with
\begin{equation*} \WJt(m_1,m_2,\ell_1,\ell_2) = \prod_{j,k \in \{1,2\}} w_t(m_j\ell_k).\end{equation*}

Using $\1_{(\ell_1,\ell_2)=1}=\sum_{d|\ell_1,d|\ell_2}\mu(d)$, and interchanging the order of summation, the above sum can be written as 

\begin{equation*}
\sum_{d\geqs 1}\frac{\mu(d)}{d^2}\sum_{\substack{m_j,\ell_j\geqs 1\\m_1 \neq m_2,\ell_1\neq\ell_2}}
\frac{e(-\alpha d (m_1-m_2)(\ell_1-\ell_2))}{m_1m_2\ell_1\ell_2}
\WJt(m_1,m_2,d\ell_1,d\ell_2),\end{equation*}
by replacing $\ell_j$ with $d\ell_j$. Now, setting $h = m_1 - m_2$, $k = \ell_1 - \ell_2$, $m = m_2$, and $\ell = \ell_2$ and spending the symmetries $m_1 \leftrightarrow m_2$ and $\ell_1 \leftrightarrow \ell_2$, this simplifies to
\begin{equation*} 4\Re\sum_{d\geqs 1}\frac{\mu(d)}{d^2}\sum_{h,k,m,\ell\geqs 1}
\frac{e(-\alpha d hk)}{m\ell(m+h)(\ell+k)}\WJt'(d,h,k,m,\ell)\end{equation*}
where 
\begin{equation*} \WJt'(d,h,k,m,\ell) = w_t(dm\ell)w_t(d(m+h)\ell)w_t(dm(\ell+k))w_t(d(m+h)(\ell+k)).\end{equation*}
Since the weights restrict $h,k,m,\ell\ll T^{1/2+\epsilon}$ the inner sum is trivially $\ll (\log T)^4$ and so we may restrict the outer sum to $d\leqs (\log T)^{4+\epsilon}$ at the cost of $O((\log T)^{-\epsilon})$.

Unfolding the integrals in the weights we find that
\begin{multline*}
\sum_{h,k,m,\ell\geqs 1}
\frac{e(-\alpha d hk)}{m\ell(m+h)(\ell+k)}\WJt(d,h,k,m,\ell)
\\
=
\frac{1}{(2\pi i )^4}\int_{(c)^4} \sum_{\substack{h,k,m,\ell\geqs 1}}
\frac{e(-\alpha d hk)}{(m+h)^{1+s_1+s_3}m^{1+s_2+s_4}(\ell+k)^{1+s_1+s_4}\ell^{1+s_2+s_3}}
\\
\times\prod_{j=1}^4\bigg(\frac{\tau}{d}\bigg)^{s_j}G(s_j)\frac{ds_j}{s_j}
\end{multline*}
where we have taken $c=1/\log T$ which is allowed by the absolute convergence of the sum in the integrand. 

Consider the sum over $h$ and $k$, initially in the region $\Re(s_j)>0$. By partial summation this is given by 
\begin{equation}
\begin{split}\label{int contin}
&
\sum_{h,k\geqs 1}\frac{e(-\alpha d hk)}{(m+h)^{1+s_1+s_3}(\ell+k)^{1+s_1+s_4}}
= 
\sum_{h\geqs 1}\frac{1}{(m+h)^{1+s_1+s_3}}\sum_{k\geqs 1}\frac{e(-\alpha d hk)}{(\ell+k)^{1+s_1+s_4}}
\\
= &
\sum_{h\geqs 1}\frac{1}{(m+h)^{1+s_1+s_3}}\cdot (1+s_1+s_4)\int_1^\infty\frac{1}{(\ell+x)^{2+s_1+s_4}}\Big(\sum_{k\leqs x}e(-\alpha d hk) \Big)dx
\\
= &
(1+s_1+s_3)(1+s_1+s_4)
\\
&\qquad\qquad\times\int_1^\infty\int_1^\infty\frac{1}{(m+y)^{2+s_1+s_3}(\ell+x)^{2+s_1+s_4}}\bigg(\sum_{\substack{h\leqs y\\k\leqs x}}e(-\alpha d hk) \bigg)dxdy.
\end{split}
\end{equation}

Now, by Lemma~\ref{expsum lem}, $\alpha$ satisfies \eqref{square root bound} for some $\delta>0$, implying that the integral in \eqref{int contin} converges absolutely for $\Re(s_1+s_3), \Re(s_1+s_4)>-\delta/8$ and represents an analytic function in this region. Further, we see that the integrand of the multiple contour integral is analytic for $\Re(s_1+s_2+s_3+s_4) > -\delta/8$ as well. 

We now shift the $s_1$ integral to the line with $\Re(s_1)=-\delta/40$, say, picking a pole up at $s_1=0$. By \eqref{int contin} and \eqref{square root bound} the integral over the new line contributes $\ll (\tau/d)^{-\delta/40}$, which results in a negligible contribution. We then progressively shift the remaining integrals to the line with $\Re(s_j)=-\delta/40$ picking up a simple pole at $s_j=0$ each time and obtaining a negligible error in the process. The residues lead to the contribution 
\begin{equation*}
c_T = \sum_{d\leqs (\log T)^{4+\epsilon}}\frac{\mu(d)}{d^2}\sum_{m,\ell\geqs 1}\frac{1}{m\ell}\int_1^\infty\int_1^\infty\frac{1}{(m+y)^{2}(\ell+x)^{2}}\Big(\sum_{h\leqs y}\sum_{k\leqs x}e(-\alpha d hk) \Big)dxdy
\end{equation*}
and so the result follows by noting that \eqref{square root bound} implies
\[
c_T \ll \sum_{d\leqs (\log T)^{4+\epsilon}} \frac{1}{d} \ll \log\log T
\]
\end{proof}

\subsection{Off-diagonals}The off-diagonal sum is given by 
\[
\mc{J}_O
=
 \sum_{n_1n_2\neq n_3n_4} \frac{e(-\alpha(n_1+n_2-n_3-n_4))}{\sqrt{n_1n_2n_3n_4}}
 \int_\mb{R}\bigg(\frac{n_1n_2}{n_3n_4}\bigg)^{it}\wjt(\ul{n})\Phi(t/T)dt.
\]
Write $h=n_1n_2-n_3n_4\neq 0$. Then similarly to \S\ref{sec: S1 prop}, we may restrict to $|h|\ll n_3n_4/T^{1-\epsilon}$ since otherwise integrating by parts using \eqref{weight diff} gives a negligible error. We then apply  the expansions
\begin{align*}
\log\bigg(\frac{n_1n_2}{n_3n_4}\bigg)
= &
\frac{h}{n_3n_4}+O(T^{-2+\epsilon})
\\
\bigg(\frac{n_1n_2}{n_3n_4}\bigg)^{it}
= &
\exp\bigg(it\frac{h}{n_3n_4}\bigg)(1+O(T^{-1+\epsilon}))
\\
n_1n_2
= &
n_3n_4(1+O(T^{-1+\epsilon})).
\end{align*}
The error terms acquired when applying these are 
\[
\ll \frac{1}{T^{1-\epsilon}}\sum_{\substack{n_1n_2-n_3n_4=h\\h\ll n_3n_4/T^{1-\epsilon}}}\frac{1}{n_3n_4} \int_\mb{R}|\wjt(\ul{n})\Phi(t/T)|dt\ll T^\epsilon
\]
where in the second inequality we have let $h,n_3,n_4$ range freely and used a divisor bound for the number of solutions of $n_1n_2=n_3n_4+h$, similar to \S\ref{sec: S1 prop}. 
Therefore 
\begin{align*}
\mc{J}_O
= &
\sum_{\substack{n_1n_2-n_3n_4=h\\h\ll n_3n_4/T^{1-\epsilon}}}
 \frac{e(-\alpha(n_1+n_2-n_3-n_4))}{n_3n_4}
 \int_\mb{R}\exp\bigg(it\frac{h}{n_3n_4}\bigg)\wjt(\ul{n})\Phi(t/T)dt
 +O(T^\epsilon).
\end{align*}


\begin{lem} We have 
\[
\mc{J}_O
 \ll 
 T\log T
 \]
 for all $\alpha\in\mathbb{R}$.
\end{lem}

\begin{proof}
Integrating by parts we acquire the main term
\[
 \frac{1}{i}\sum_{\substack{n_1n_2-n_3n_4=h\\h\ll n_3n_4/T^{1-\epsilon}}} \frac{e(-\alpha(n_1+n_2-n_3-n_4))}{h}\int_\mathbb{R}\exp\bigg(it\frac{h}{n_3n_4}\bigg) \frac{d}{dt}\big(\wjt(\ul{n})\Phi(t/T)\big)dt
\]
Let us first consider the term involving the derivative of $\Phi(t/T)$. Pushing the sum through the integral we acquire the sum
\begin{align*}
&
\sum_{0\neq h \ll T^{\epsilon}}\frac{1}{h}\sum_{\substack{n_1n_2-n_3n_4=h\\n_3n_4\gg |h|T^{1-\epsilon}}} {e(-\alpha(n_1+n_2-n_3-n_4))}\exp\bigg(it\frac{h}{n_3n_4}\bigg)\wjt(\ul{n}) 
\\
\ll &
\sum_{0\neq h \ll T^{\epsilon}}\frac{1}{|h|}\sum_{\substack{n_1n_2-n_3n_4=h\\n_3n_4\gg |h|T^{1-\epsilon}}} |w_t(n_1)w_t(n_2)w_t(n_3)w_t(n_4)|. 
\end{align*}
We remove the variable $n_4$ by writing this as 
\begin{align*}
 &
\sum_{0\neq |h|\ll T^\epsilon}\frac{1}{|h|}\sum_{n_2,n_3}|w_t(n_2)w_t(n_3)|\sum_{\substack{n_1n_2\equiv h \!\!\!\mod n_3\\ n_1n_2\gg |h|T^{1-\epsilon}+h}}|w_t(n_1)w_t(\tfrac{n_1n_2-h}{n_3})|
\\
= &
\sum_g\sum_{0\neq |h|\ll T^\epsilon}\frac{1}{|h|}\sum_{(n_2,n_3)=1}|w_t(gn_2)w_t(gn_3)|\sum_{\substack{gn_1n_2\equiv h \!\!\!\mod gn_3\\ gn_1n_2\gg |h|T^{1-\epsilon}+h}}|w_t(n_1)w_t(\tfrac{gn_1n_2-h}{gn_3})|.
\end{align*}
The inner sum is empty unless $g\mid h$ and hence this becomes
\[
\sum_g\frac{1}{g}\sum_{0\neq |h|\ll T^\epsilon/g}\frac{1}{|h|}\sum_{(n_2,n_3)=1}|w_t(gn_2)w_t(gn_3)|\sum_{\substack{n_1\equiv \ol{n_2}h \!\!\!\mod n_3\\ n_1n_2\gg |h|T^{1-\epsilon}+h}}|w_t(n_1)w_t(\tfrac{n_1n_2-h}{n_3})|
\]
where we have cancelled a factor of $g$ in the congruence condition and inverted $n_2$ modulo $n_3$.

Consider the sum with $n_2<n_3$. By \eqref{weight bound} this gives a contribution
\begin{align*}
\ll &
 \sum_g\frac{1}{g}\sum_{0\neq |h|\ll T^\epsilon/g}\frac{1}{|h|}\sum_{n_3}|w_t(gn_3)|\sum_{n_2<n_3}\sum_{\substack{n_1\equiv \ol{n_2}h \!\!\!\mod n_3}}|w_t(n_1)|
\\
\ll &
\sum_g\frac{1}{g}\sum_{0\neq |h|\ll T^\epsilon/g}\frac{1}{|h|}\sum_{n_3}|w_t(gn_3)|\sum_{n_2<n_3}
\sum_{\substack{n_1\equiv \ol{n_2}h \!\!\!\mod n_3}}\big(\mathds{1}_{n_1\leqs \tau}+\mathds{1}_{n_1> \tau}\big(\tfrac{\tau}{n_1}\big)^2\big)
\\
\ll &
\,\,\tau\sum_g\frac{1}{g}\sum_{0\neq |h|\ll T^\epsilon/g}\frac{1}{|h|}\sum_{n_3}\big(\mathds{1}_{n_3\leqs \tau/g}+\mathds{1}_{n_3> \tau/g}\big(\tfrac{\tau}{gn_3}\big)^2\big)
\\
\ll &
\,\,T\log T
\end{align*}
Similarly, the case of $n_2>n_3$ can be bounded by
\begin{align*}
\ll &
 \sum_g\frac{1}{g}\sum_{0\neq |h|\ll T^\epsilon/g}\frac{1}{|h|}\sum_{n_3}|w_t(gn_3)|\sum_{n_2>n_3}|w_t(gn_2)|\sum_{\substack{n_1\equiv \ol{n_2}h \!\!\!\mod n_3}}|w_t(\tfrac{n_1n_2-h}{n_3})|
\\
\ll &
\sum_g\frac{1}{g}\sum_{0\neq |h|\ll T^\epsilon/g}\frac{1}{|h|}\sum_{n_3}|w_t(gn_3)|\sum_{n_2>n_3}\big(\mathds{1}_{n_2\leqs \tau/g}+\mathds{1}_{n_2> \tau/g}(\tfrac{\tau}{gn_2})^2\big)
\cdot(\tfrac{\tau}{n_2})
\\
\ll &
\,\,\tau\sum_g\frac{1}{g}\sum_{0\neq |h|\ll T^\epsilon/g}\frac{1}{|h|}\sum_{n_3}\big(\mathds{1}_{n_3\leqs \tau/g}\big(\log(\tfrac{\tau}{gn_3})+1\big) +\mathds{1}_{n_3> \tau/g}\big(\tfrac{\tau}{gn_3}\big)^2\big)
\\
\ll &
\,\,T\log T
\end{align*}
where the first line follows from applying \eqref{weight bound}, since $(\tau n_3+h)/n_2=(1+o(1))\tau n_3/n_2$, and the last line follows from the fact that $\sum_{m\leqs x}\log(x/m)\ll x$. Thus, for the term involving the derivative of $\Phi(t/T)$ we acquire a contribution
\[
\ll \int_\mb{R} \frac{1}{T}|\Phi^\prime(t/T)|\cdot T\log T dt\ll T\log T.
\]
For the terms involving $d/dt(w_t(n_j))$ we apply \eqref{weight diff} appropriately. This gives us the requisite factor of $T^{-1}$ along with the same quality bounds that have been applied for $|w_t(n_j)|$ throughout. The result then follows. 
\end{proof}


\section{Concluding Remarks} \label{sec: issues}

We outline here some difficulties in upgrading Theorem~\ref{main thm} to an asymptotic. 
The arguments from \S\S\ref{sec: S1 prop}-\ref{sec: S2 prop} can be adapted to prove the following proposition. 

 \begin{prop}\label{S1S2 prop} Under the same conditions as Proposition~\ref{S1 prop},   
 \[
 \int_\mathbb{R}\bigg|\sum_{n\geqs 0}\frac{w_t(n+\alpha)}{(n+\alpha)^{1/2+it}}\bigg|^2\bigg|\sum_{n\geqs 1}\frac{e(-n\alpha)}{n^{1/2-it}}w_t(n)\bigg|^2\Phi(t/T)dt
 =
\frac{\hat{\Phi}(0)}{4}T(\log T)^2+O(T(\log T)^{5/3}).
 \]
 \end{prop}
In fact, the proof of this is slightly more straightforward than that of Proposition~\ref{S1 prop}, since the analogue of \eqref{h_i} is simpler,
\begin{equation*} \label{h_i2}
h_1=n_1n_2-n_3n_4,\qquad h_2=n_2-n_4.
\end{equation*}
Since Proposition~\ref{S1 prop} and \ref{S2 prop} are sufficient for Theorem~\ref{main thm} due to the first application of H\"older's inequality in \eqref{wasteful holder}, however, we leave the details to the reader.

On combining Propositions~\ref{S1 prop}, \ref{S2 prop}, and \ref{S1S2 prop} with \eqref{decomposition}, we see that one could prove an asymptotic of the shape
\[ \int_\R |\zeta(\tfrac12+it,\alpha)|^4 \Phi(t/T) dt \sim 2\hat{\Phi}(0)  T(\log T)^2, \]
from which the unsmoothed fourth moment could be deduced, provided one could show
\[ \int_\mathbb{R}\bigg(
4\Re S_1^2\ol{S_1S_2}
+2\Re S_1^2\ol{S_2}^2
+4\Re S_1S_2\ol{S_2}^2
\bigg)\Phi(t/T)dt =o(T(\log T)^2).\]
It is well known that achieving the required cancellation from such terms can be difficult due to the presence of the $\chi^2$ factors from the functional equation. Their rapid oscillation leads to long sums after applying the method of stationary phase, and these require very precise control with respect to the Diophantine properties of $\alpha$. Equivalently, one may apply a Heath-Brown \cite{HB} type approximate functional equation in which there are no cross terms or $\chi$ factors, but in which the sums are longer, and then one encounters the same difficulties in the off-diagonals. As a compromise, we work with shorter sums but must sacrifice an asymptotic when applying H\"older's inequality in \eqref{wasteful holder}.

\end{document}